\documentclass[11pt]{amsart}%
\usepackage{amsmath}
\usepackage[cmtip,all]{xy}%
\usepackage{amssymb}%
\usepackage{amscd}%
\usepackage{amsxtra}
\usepackage{color}%
\usepackage[active]{srcltx}
\usepackage{eucal}%
\usepackage{mathrsfs}%
\usepackage[all]{xy}%
\usepackage{paralist}
\def\ol#1{\overline{#1}}
\def\wh#1{\widehat{#1}}
\def\wt#1{\widetilde{#1}}
\def\ul#1{\underline{#1}}
\theoremstyle{plain}
    \newtheorem{theorem}{Theorem}[subsection]
    
    \newtheorem{proposition}[theorem]{Proposition}
    \newtheorem{lemma}[theorem]{Lemma}
    \newtheorem{corollary}[theorem]{Corollary}
\theoremstyle{definition}
    \newtheorem{definition}[theorem]{Definition}

    \newtheorem{remark}[theorem]{Remark}

\numberwithin{equation}{subsection}
\def\Alphabet{1,A,B,C,D,E,F,G,H,I,J,K,L,M,N,O,P,Q,R,S,T,U,V,W,X,Y,Z}
\def\alphabet{a,b,c,d,e,f,g,h,i,j,k,l,m,n,o,p,q,r,s,t,u,v,w,x,y,z}
\def\endpiece{xxx}
\def\makeAlphabet[#1]{\expandafter\makeA#1,xxx,}
\def\makealphabet[#1]{\expandafter\makea#1,xxx,}
\def\makeA#1,{\def\temp{#1}\ifx\temp\endpiece\else%
\mkbb{#1}\mkfrak{#1}\mkbf{#1}\mkcal{#1}\mkscr{#1}\expandafter\makeA\fi}%
\def\makea#1,{\def\temp{#1}\ifx\temp\endpiece\else\mkfrak{#1}\mkbf{#1}\expandafter\makea\fi}%
\def\mkbb#1{\expandafter\def\csname bb#1\endcsname{\mathbb{#1}}}
\def\mkfrak#1{\expandafter\def\csname fr#1\endcsname{\mathfrak{#1}}}
\def\mkbf#1{\expandafter\def\csname b#1\endcsname{\mathbf{#1}}}
\def\mkcal#1{\expandafter\def\csname c#1\endcsname{\mathcal{#1}}}
\def\mkscr#1{\expandafter\def\csname s#1\endcsname{\mathscr{#1}}}
\def\makeop[#1]{\xmakeop#1,xxx,}
\def\mkop#1{\expandafter\def\csname #1\endcsname{{\mathrm{#1}}}} %
\def\xmakeop#1,{\def\temp{#1}\ifx\temp\endpiece\else\mkop{#1}\expandafter\xmakeop\fi}%
\makeAlphabet[\Alphabet]
\makealphabet[\alphabet]
\makeop[Alt,End,Ext,Hom,Sym,Tor,Aut,Det,Ind,TS,TSym]
\makeop[CH,Pic,div,Div]
\makeop[Ker,Im,Coker,Coim,id,pr,proj,ker,coker,im,cor]
\makeop[Spec,Spf,Spm,Isom]
\makeop[Re,Im,Min,Max]%
\makeop[dR,Nis,crys,cris,rig,syn,tor,et,Cont,cont]
\makeop[Gal,Res,ab,res,sgn,R,Per]
\makeop[pol,Iw,cycl,Exp,Der,mom,comp]
\makeop[Var,an,Isoc,univ,Cusp]%
\makeop[Lie,coLie,MIC,DR,Gr,Ab,Cone,ord,Frob,N,aH]
\makeop[Meas,arith,mot,can,tors,conv,contr,opp,mult,const,para,near,naive,taut]
\makeop[Gl,Sl]
\makeop[Eis,tr,Fil,gr,sp,Tate,rk,fin,cl,red,Dir,ev]
\makeop[Gl,Sl,GL,SL]

\def\M0{M^0}
\def\sLog{\sL\!\!\operatorname{og}}
\def\Log{\sLog}

\def\Gm{{\bbG_m}}
\def\isom{\cong}

\def\verk{\circ}
\def\prolim{\varprojlim}

\def\Isocda{\Isoc^{\kern-0.5mm\dagger}}

\def\isom{\cong}
\def\verk{\circ}

\def\prolim{\varprojlim}

\def\b1{\boldsymbol{1}}

\def\Cont{{\cC}}

\begin{document}

\title[Eisenstein classes and elliptic Soul\'e elements]{
Eisenstein classes, elliptic Soul\'e elements 
and the $\ell$-adic  elliptic polylogarithm}
\begin{abstract}
In this paper we study systematically the $\ell$-adic realization of
the elliptic polylogarithm  in the context of
sheaves of Iwasawa modules. This leads to a description of
the elliptic polylogarithm in terms of elliptic units. 
As an application we prove a precise 
relation between $\ell$-adic Eisenstein classes and elliptic 
Soul\'e elements. This allows to give a new proof of the formula
for the residue of the $\ell$-adic Eisenstein classes
at the cusps and the formula for the cup-product construction in \cite{Huber-Kings99},
which relies only on the explicit description of elliptic units.
This computation is the main input in the proof of Bloch-Kato's compatibility
conjecture 6.2. needed in the proof of Tamagawa number conjecture for
the Riemann zeta function.
\end{abstract}
\author{Guido Kings}
\address{Fakult\"at f\"ur Mathematik\\
Universit\"at Regensburg\\
93040 Regensburg\\
Germany}
\maketitle

\section*{Introduction}
The purpose of this paper is twofold: on the one hand we prove a new
and precise relation between $\ell$-adic Eisenstein classes and 
elliptic Soul\'e elements using a description of the integral $\ell$-adic
elliptic polylogarithm in terms of elliptic units. On the other hand this
relation will be used to give a new proof for the
cup-product construction formula, which is the main result of  \cite{Huber-Kings99} and
is the main input in \cite{Huber-pune} to obtain a proof of 
Bloch-Kato's compatibility conjecture 6.2. This new
proof uses only elementary properties of elliptic units.

The explicit description of the integral $\ell$-adic elliptic polylogarithm
in terms of elliptic units was already one of the main results in the paper
\cite{Kings01}. There we used an approach via one-motives to treat the 
logarithm sheaf. But 
the main application of the $\ell$-adic elliptic polylogarithm
is in the context of Iwasawa theory, which makes it desirable
to approach the elliptic polylogarithm 
systematically in this context. That such an approach
is possible, is already suggested in the ground-braking paper \cite{BeLe}.

In Iwasawa theory
Kato, Perrin-Riou and Colmez pointed out the usefulness to work with ``Iwasawa
cohomology'', which is continuous Galois cohomology with values in an 
Iwasawa algebra. We generalize
this idea to treat families of Iwasawa modules under a family of 
Iwasawa algebras.
The main example for this is the family of Iwasawa algebras on the moduli
scheme of elliptic curves, where one has in each fibre the Iwasawa algebra of
the Tate module of the corresponding elliptic curve.

It is the fundamental idea of
Soul\'e \cite{Soule-p-adic-regulators} that twisting of units can be used
to produce interesting cohomology classes.  Already in Kato's paper
\cite{Kato-Iwasawa-theory-p-adic-Hodge} it is implicit
that this twisting is related to 
the Iwasawa cohomology. Later Colmez used this explicitly in
\cite{Colmez-Iwasawa-repr-de-Rham}, where he used moment maps of 
$\bbQ_\ell$-measure algebras. For our approach it is crucial to 
develop this further by constructing the moment map at finite
level. We show that in the cyclotomic case one obtains the elements defined
and studied by Soul\'e and Deligne. Work by Soul\'e in the CM elliptic case 
and Kato's work in \cite{Kato-p-adic} suggest
that one should carry out Soul\'e's twisting construction also in the
modular curve case to obtain
elliptic Soul\'e elements. One of the main results in this paper is that
these elliptic Soul\'e elements are essentially the 
$\ell$-adic Eisenstein classes in \cite{Huber-Kings99}.

With the general theory of sheaves of Iwasawa modules, we obtain a concrete 
description of the elliptic polylogarithm in terms of the norm compatible
elliptic units defined and studied by Kato \cite{Kato-p-adic}. This gives strong
ties of the elliptic polylogarithm to recent developments in Iwasawa theory and
also allows many explicit computations with the $\ell$-adic elliptic polylogarithm. 

As an application of the concrete description of the elliptic polylogarithm,
we give a new proof of the residue
computation for $\ell$-adic Eisenstein classes 
on the moduli scheme for elliptic curves 
(Corollary \ref{eis-residue}).

A second application is the evaluation of the cup-product construction
used in \cite{Huber-Kings99} (and explained in this volume in \cite{Huber-pune})
to obtain  elements in the motivic cohomology 
of cyclotomic fields and to prove Conjecture 6.2 in \cite{Bloch-Kato-TNC}. 
The approach taken here, does not need any computations
of the cyclotomic polylogarithm as in \cite{Huber-Kings99}. It relies only
on  the concrete evaluation of the elliptic units at the cusps.

An overview of the main results in this paper is given in
Section \ref{main-results}.

\textbf{Acknowledgements:} It is a pleasure to thank the organizers of 
the Pune workshop for invitation and the audience 
for their interest and their questions. 
Further I would like to thank several people for helpful comments on
an earlier version of this paper:
Annette Huber's detailed reading
led to many improvements and corrections.
Ren\'e Scheider pointed out a missing factor in a formula.
The  referee pointed out several improvements in the exposition
and suggested 
to work with symmetric tensors in the construction
of the moment map. This led to a complete rewriting of the earlier version.
\tableofcontents

%
\section*{Notations}
%
We fix an integer $N\ge 3$ and let $Y(N)$ the moduli space of 
elliptic curves $\cE$ with a full level $N$-structure 
$\alpha:(\bbZ/N\bbZ)^2\isom \cE[N]$. We denote by 
$$
\pi:\cE\to Y(N)
$$
the universal elliptic curve. 
We let $X(N)$ be the smooth compactification of $Y(N)$ and denote
by $j:Y(N)\hookrightarrow X(N)$ the open immersion. If we fix an $N$-th
root of unity $\zeta_N:=e^{2\pi i/N}\in \bbC$ and consider the Tate 
curve $\cE_q$ 
with the level structure $\alpha:(\bbZ/N\bbZ)^2\to \cE_q[N]$ given
by $(a,b)\mapsto q^a\zeta_N^b$. This induces a map of schemes
$$
\Spec\bbQ(\zeta_N)((q^{1/N}))\to Y(N),
$$
which extends to 
$\Spec\bbQ(\zeta_N)[[q^{1/N}]]\to X(N)$ and a hence a map
$\infty:\Spec\bbQ(\zeta_N)\to X(N)$, whose image we call the cusp $\infty$. 

Define \'etale sheaves on $Y(N)$ by
\begin{align}
\begin{split}
\sH_r&:=(R^1\pi_*\bbZ/\ell^r\bbZ)^\vee\isom R^1\pi_*\bbZ/\ell^r\bbZ(1)\\
\sH&:=(R^1\pi_*\bbZ_\ell)^\vee\isom R^1\pi_*\bbZ_\ell(1)\\
\sH_{\bbQ_\ell}&:=(R^1\pi_*\bbQ_\ell)^\vee\isom R^1\pi_*\bbQ_\ell(1)\\
\end{split}
\end{align}
where $(.)^\vee$ denotes the $\bbZ/\ell^r\bbZ$, $\bbZ_\ell$ and
$\bbQ_\ell$ dual respectively. We denote by $\Sym^k \sH_r$, $\Sym^k\sH$
and $\Sym^k\sH_{\bbQ_\ell}$ the $k$-th symmetric power as 
$\bbZ/\ell^r\bbZ$-, $\bbZ_\ell$- and
$\bbQ_\ell$-modules respectively. In the same way we denote by 
$\TSym^k\sH_r$, $\TSym^k\sH$ and $\TSym^k\sH_{\bbQ_\ell}$
the functor  of symmetric $k$-tensors as
$\bbZ/\ell^r\bbZ$-, $\bbZ_\ell$- and
$\bbQ_\ell$-modules respectively. Note that there is a canonical 
map 
\begin{equation}
\Sym^k\sH_r\to \TSym^k\sH_r,
\end{equation}
which extends to a homomorphism of graded algebras 
$\Sym^\cdot\sH_r\to \TSym^\cdot\sH_r$
and similarly for $\sH$ and $\sH_{\bbQ_\ell}$.

As we will not only deal with $\ell$-adic sheaves, we work in the
bigger abelian category of inverse systems $\sF=(\sF_r)_{r\ge 1}$ of
\'etale sheaves modulo Mittag-Leffler-zero systems (which means to 
work in the pro-category) and define the continuous \'etale
cohomology in the sense of \cite{Jannsen-cont-et-coh}. This means that
$H^i(S,\sF)$
is the $i$-th derived functor of 
$$
\sF\mapsto \prolim_rH^0(S,\sF_r).
$$ 
More generally, one defines 
$$
R^i\pi_*\sF
$$
for a morphism $\pi:S\to T$ to be the $i$-th derived functor of 
$\sF\mapsto \prolim_r\pi_*\sF_r$. For $\ell$-adic sheaves, we also
consider  $\Ext$-groups
$$
\Ext^i_S(\sF,\sG),
$$
which are the right derived functors of $\Hom_S(\sF,-)$. 

Of crucial
importance is the following lemma:
\begin{lemma}\label{h-1-prolim}
Let $\sF=(\sF_r)_{r\ge 1}$ be a projective system with $H^0(S,\sF_r)$ finite, then
$$
H^1(S,\sF)=\prolim_rH^1(S,\sF_r).
$$
\end{lemma}
\begin{proof}
This follows from \cite[Lemma 1.15, Equation (3.1)]{Jannsen-cont-et-coh} as the $H^0(S,\sF_r)$ satisfy the Mittag-Leffler condition.
\end{proof}

The 
quotient category of the $\ell$-adic sheaves (or $\bbZ_\ell$-sheaves)
by the torsion sheaves is the category of $\bbQ_\ell$-sheaves.
In the case 
of projective systems of $\bbQ_\ell$-sheaves $\sF=(\sF_r)_{r\ge 0}$ we
use the ad hoc definitions
\begin{align}\label{coh-of-proj-sys-def}
\begin{split}
H^i(S,\sF)&:=\prolim_r H^i(S,\sF_r)\\
\Ext^i_S(\sG,\sF)&:=\prolim_r \Ext^i_S(\sG,\sF_r),
\end{split}
\end{align}
where $\sG$ is just a $\bbQ_\ell$-sheaf.

%
%
\section{Statement of the main results}\label{main-results}
%
%
For better orientation of the reader we give an overview of the
main results in this paper and the strategy and main ingredients 
of the proof.
%
\subsection{The residue at $\infty$ of the Eisenstein class}
%

We identify sections $t:Y(N)\to \cE$ with elements in $(\bbZ/N\bbZ)^2$
via the universal level-$N$-structure on $\cE$.
For any $(0,0)\neq t  \in (\bbZ/N\bbZ)^2$ one can 
define a so called Eisenstein class
$$
\Eis_{\bbQ_\ell}^k(t)\in H^1(Y(N),\Sym^k\sH_{\bbQ_\ell}(1))
$$
(cf. Definition \ref{Eis-def}). It is convenient to 
introduce the following notation:
For any map 
$\psi:(\bbZ/N\bbZ)^2\setminus\{(0,0)\}\to \bbQ_\ell$
we put
$$
\Eis_{\bbQ_\ell}^k(\psi):=\sum_{t\neq e}\psi(t)\Eis_{\bbQ_\ell}^k(t).
$$
It is shown in \cite{Blasius-pune}
that $ \Eis_{\bbQ_\ell}^k(\psi)$ is in fact the image of a class
in motivic cohomology under the regulator map. We are interested in
the image of $\Eis_{\bbQ_\ell}^k(\psi)$ under the residue map
$$
\res_\infty:H^1(Y(N),\Sym^k\sH_{\bbQ_\ell}(1))\to H^0(\infty, \bbQ_\ell)\isom \bbQ_\ell
$$
as defined in Definition \ref{residue-Qell-def}.
The following result was first proved in \cite{BeLe} by a completely
different method:
\begin{theorem}[See Corollary \ref{eis-residue}]
One has
$$
\res_{\infty}(\Eis^k_{\bbQ_\ell}(\psi))=
\frac{-N^k}{(k+2)k!}\sum_{(a,b)\in (\bbZ/N\bbZ)^2\setminus\{(0,0)\}}\psi(a,b)
B_{k+2}(\{\frac{a}{N}\}),
$$ 
where $B_{k+2}$ denotes the $k+2$ Bernoulli polynomial and 
$\{\frac{a}{N}\}$ is the representative in $[0,1[$ of $\frac{a}{N}$.
\end{theorem}

\subsection{Evaluation of the cup-product construction}
For two maps
$$
\phi,\psi:(\bbZ/N\bbZ)^2\setminus\{(0,0)\}\to \bbQ_\ell 
$$ 
we can consider the
cup-product
\begin{equation*}
\Eis_{\bbQ_\ell}^k(\phi)\cup \Eis_{\bbQ_\ell}^k(\psi)\in
H^2(Y(N),\Sym^k\sH_{\bbQ_\ell}\otimes \Sym^k\sH_{\bbQ_\ell}(2)).
\end{equation*}
The cup-product pairing $\sH_{\bbQ_\ell}\otimes \sH_{\bbQ_\ell}\to \bbQ_\ell(1)$ 
induces a pairing
$$
\Sym^k\sH_{\bbQ_\ell}\otimes \Sym^k\sH_{\bbQ_\ell}\to \bbQ_\ell(k)
$$
and we can consider the image of $\Eis_{\bbQ_\ell}^k(\psi)\cup \Eis_{\bbQ_\ell}^k(\phi)$ in $H^2(Y(N),\bbQ_\ell(k+2))$.
$$
\Eis_{\bbQ_\ell}^k(\psi)\cup \Eis_{\bbQ_\ell}^k(\phi)\in
H^2(Y(N),\bbQ_\ell(k+2)).
$$
Let 
$$
\res_\infty:H^2(Y(N),\bbQ_\ell(k+2))\to H^1(\infty,\bbQ_\ell(k+1))
$$
be the edge morphism in the Leray spectral sequence for $Rj_*$
using the isomorphism $\infty^*R^1j_*\bbQ_\ell(k+2)\isom \bbQ_\ell(k+1)$.
\begin{definition}
Let $\phi_\infty,\psi:(\bbZ/N\bbZ)^2\setminus\{(0,0)\}\to \bbQ_\ell $ be two 
maps and suppose
that $\res_\infty(\Eis_{\bbQ_\ell}^k(\phi_\infty))=1$ and
$\res_\infty(\Eis_{\bbQ_\ell}^k(\psi))=0$.
Then 
$$
\Dir_\ell(\psi):=\res_\infty(\Eis_{\bbQ_\ell}^k(\psi)\cup
\Eis_{\bbQ_\ell}^k(\phi_\infty))\in  H^1(\infty,\bbQ_\ell(k+1))
$$
is called the \emph{cup-product construction} 
(compare \cite[Definition 4.1.3.]{Huber-pune}).
\end{definition}
Note that  $\Dir_\ell(\psi)$
does not depend on the choice of $\phi_\infty$
(as follows from the formula in Theorem \ref{Dir-as-evaluation}).
The main result of this paper is:
\begin{theorem}[see Corollary \ref{Dir-formula}]\label{Eis-at-infty}
Let  $\psi:(\bbZ/N\bbZ)^2\setminus\{(0,0)\}\to \bbQ_\ell $ be a map such that
$$
\res_\infty(\Eis_{\bbQ_\ell}^k(\psi))=0.
$$
Then one has
$$
\Dir_\ell(\psi))=\frac{-1}{Nk!}\sum_{0\neq b\in\bbZ/N\bbZ}
\psi(0,b)\wt c_{k+1}(\zeta_N^b)\in H^1(\infty,\bbQ_\ell(k+1)),
$$
where $\wt c_{k+1}(\zeta_N^b)$ is the modified cyclotomic Soul\'e-Deligne element
from Definition \ref{cyc-soule-def}.
\end{theorem}
It is explained in \cite{Huber-pune} how this theorem settles the
compatibility conjecture 6.2. in \cite{Bloch-Kato-TNC}.

The main idea in this paper (building upon our former work \cite{Kings01}) 
is to describe  a $\bbZ_\ell$-version of $\Eis_{\bbQ_\ell}^k(t)$ as
Soul\'e's twisting construction applied to elliptic units.
Then all explicit computations with the Eisenstein classes are reduced
to computations with the elliptic units. 

%
\subsection{Eisenstein classes and elliptic units}
%

We explain how the Eisenstein classes are related to elliptic 
units and in particular how one can define these classes integrally. 
For this we need to introduce some notation. 

Recall that the
Eisenstein class is associated to a non-zero $N$-torsion section 
$t:Y(N)\to \cE[N]$. Let
$\ell$ be a prime number. 
We define the $\cE[\ell^r]$-torsor $\cE[\ell^r]\langle t\rangle$ on
the modular curve $Y(N)$
by the cartesian diagram
$$
\begin{CD}
\cE[\ell^r]\langle t\rangle@>>> \cE[\ell^rN]\\
@Vp_{r,t} VV@VV[\ell^r]V\\
Y(N)@>t>>\cE[N].
\end{CD}
$$
\begin{definition}
Define the \'etale sheaf $\Lambda_r(\sH_r\langle t\rangle)$ on $Y(N)$ by
$$
\Lambda_r(\sH_r\langle t\rangle):=p_{r,t*}\bbZ/\ell^r\bbZ.
$$
If $t=e$ is the identity section we write
$\Lambda_r(\sH_r)$. 
\end{definition}
The sheaves $\Lambda_r(\sH_r\langle t\rangle)$ form an inverse 
system with respect to the trace map 
$$
\Lambda_{r+1}(\sH_{r+1}\langle t\rangle)\to
\Lambda_r(\sH_r\langle t\rangle)
$$
and we denote the resulting pro-system by 
\begin{equation}
\Lambda(\sH\langle t\rangle):=(\Lambda_r(\sH_r\langle t\rangle))_{r\ge 1}.
\end{equation}
For $t=e$ we write $\Lambda(\sH):=(\Lambda_r(\sH_r))_{r\ge 1}$.
\begin{remark}
The sheaves $\Lambda(\sH\langle t\rangle)$ form the 
main example of sheaves of Iwasawa modules mentioned in the title
of this paper. The connection is explained in Lemma 
\ref{relation-to-Iwasawa-modules}.
\end{remark}
Note that by Lemma \ref{h-1-prolim} we have
\begin{align}
\begin{split}
H^1(Y(N),\Lambda(\sH\langle t\rangle)(1))&\isom
\prolim_r H^1(Y(N),\Lambda_r(\sH_r\langle t\rangle)(1))\\
&\isom
\prolim_r H^1(\cE[\ell^r]\langle t\rangle,\bbZ/\ell^r\bbZ(1)).
\end{split}
\end{align}
Fix an auxiliary integer $c>1$, which is prime to $6\ell N$. Then
Kato has defined a norm-compatible unit 
${_c\vartheta_\cE}$ on $\cE\setminus \cE[c]$
(cf. Theorem \ref{modular-units}).
Note that for an $N$-torsion point $t\neq e$ one has 
$\cE[\ell^r]\langle t\rangle\subset \cE\setminus \cE[c]$ by our condition on $c$.
Thus, we can restrict $ {_c\vartheta_\cE}$ to an invertible function on 
$\cE[\ell^r]\langle t\rangle$.
The Kummer map (see \ref{Kummer-def}) gives a class
$$
\cE\cS^{\langle t\rangle}_{c,r}:=\partial_r({_c\vartheta_\cE})\in H^1(\cE[\ell^r]\langle t\rangle,\bbZ/\ell^r\bbZ(1))
$$
and by the norm-compatibility we can define:
\begin{definition}
Let 
$$
\cE\cS^{\langle t\rangle}_c:=\prolim_r \partial_r({_c\vartheta_\cE})\in
H^1(S, \Lambda(\sH\langle t\rangle)(1)).
$$
\end{definition}
In \ref{sheaf-mom-def} we define a moment map
$$
\mom^k_{t}:\Lambda(\sH)\to \TSym^k\sH
$$
which gives rise to a map
\begin{equation}\label{comparison-map}
H^1(S, \Lambda(\sH\langle t\rangle)(1))\xrightarrow{\mom^k_{t}}
H^1(S, \TSym^k\sH(1)).
\end{equation}
We show in Proposition \ref{mom-and-twist}, inspired by a result of Colmez \cite{Colmez-Iwasawa-repr-de-Rham}:
\begin{proposition}[see \ref{mom-and-twist}]
The element 
$$
_ce_k(t):=\mom^k_t(\cE\cS^{\langle t\rangle}_c)\in
H^1(S, \TSym^k\sH(1))
$$
coincides with Soul\'e's twisting construction (see \ref{Soule-twisting})
applied to the norm compatible elliptic units $_c\vartheta_\cE$
and is called the \emph{elliptic Soul\'e element}.
\end{proposition}
Consider the image of $\mom^k_t(\cE\cS^{\langle t\rangle}_c)$
in $H^1(S, \TSym^k\sH_{\bbQ_\ell}(1))$. 
The isomorphism $\Sym^k\sH_{\bbQ_\ell}\to \TSym^k\sH_{\bbQ_\ell}$
induces
\begin{equation}\label{comparison-iso}
H^1(S, \Sym^k\sH_{\bbQ_\ell}(1))\isom 
H^1(S, \TSym^k\sH_{\bbQ_\ell}(1)),
\end{equation}
which allows us to consider 
$$
\wt\mom^k_t(\cE\cS^{\langle t\rangle}_c):=
\frac{1}{N^k}\mom^k_t(\cE\cS^{\langle t\rangle}_c)\in 
H^1(S, \Sym^k\sH_{\bbQ_\ell}(1)).
$$
\begin{theorem}[see Theorem \ref{pol-eis}]
With the above notation the equality 
$$
\frac{1}{N^k}{_ce_k(t)}=
\wt\mom^k_t(\cE\cS^{\langle t\rangle}_c)=
\frac{-1}{N^{k-1}}(c^2\Eis_{\bbQ_\ell}^k(t)-c^{-k}\Eis_{\bbQ_\ell}^k({[c]t}))
$$
holds in 
$H^1(S, \Sym^k\sH_{\bbQ_\ell}(1))$.
In particular, if $c\equiv 1 \mod{N}$ one has 
$$
{_ce_k(t)}=
-N(c^2-c^{-k})\Eis_{\bbQ_\ell}^k(t).
$$
\end{theorem}
This is the desired relation between $\Eis_{\bbQ_\ell}^k(t)$ and
the elliptic Soul\'e element.

%
%
\section{Sheaves of Iwasawa modules and the moment map}
%
%
In this section we consider sheaves of Iwasawa modules
and define the moment map. As a motivation we start
by looking at the case of modules under the Iwasawa algebra.
%
\subsection{Iwasawa algebras}
%

Fix a prime number $\ell$. Let $X$ 
be a totally disconnected compact topological space of the form
$$
X=\prolim_r X_r
$$ 
with $X_r$ finite discrete.
We denote by 
\begin{equation*}
\Cont(X,\bbZ_\ell):=\{f:X\to \bbZ_\ell\mid f\mbox{ continuous}\}
\end{equation*}
the continuous $\bbZ_\ell$-valued functions on $X$ together with
the sup-norm $||-||_\infty$.
\begin{definition}
The space of $\bbZ_\ell$-valued measures on $X$ is
$$
\Lambda(X):=\Hom_{\bbZ_\ell}(\Cont(X,\bbZ_\ell), \bbZ_\ell).
$$
We also write 
$\Lambda_r(X):=\Hom_{\bbZ_\ell}(\Cont(X,\bbZ_\ell), \bbZ/\ell^r\bbZ)$
for the $\bbZ/\ell^r\bbZ$-valued measures on $X$.
For each $\mu\in \Lambda(X)$ we write 
$$
\int_Xf\mu:=\mu(f)
$$
and for $x\in X$ we let $\delta_x\in \Lambda(X)$ be the Dirac distribution
characterized by $\delta_x(f)=f(x)$.
\end{definition}
As every continuous function in $\Cont(X,\bbZ_\ell)$ is the 
uniform limit of locally constant functions, we have
$$
\Lambda(X)=\prolim_r\Hom_{\bbZ_\ell}(\Cont(X_r,\bbZ_\ell),\bbZ_\ell)
=\prolim_r\Lambda(X_r)=\prolim_r\Lambda_r(X_r).
$$
For a continuous map  $\phi:X\to Y$ one has a homomorphism
\begin{equation}
\phi_!:\Lambda(X)\to \Lambda(Y)
\end{equation}
defined by $(\phi_!\mu)(f):=\mu(f\verk \phi)$. If $U\subset X$ is
open compact one has  
$$
\Lambda(X)\isom \Lambda(U)\oplus \Lambda(X\setminus U).
$$
Define 
$$
\Lambda(X)\wh\otimes\Lambda(Y):=\prolim_r (\Lambda(X_r)\otimes_{\bbZ_\ell}
\Lambda(Y_r))
$$
then one has a canonical isomorphism
$$
\Lambda(X\times Y)\isom \Lambda(X)\wh\otimes\Lambda(Y).
$$

Let now $X=H=\prolim_r H_r$ be a profinite group. Then $\Lambda(H_r)$ 
is the group algebra of $H_r$ and $\Lambda(H)$ inherits a $\bbZ_\ell$-algebra
structure. This algebra structure can also be defined directly by using the 
convolution of measures 
\begin{equation*}
\mu*\nu:=\mult_!(\mu\otimes \nu),
\end{equation*} 
where $\mult:H\times H\to H$ is the group multiplication. As 
$\delta_g*\delta_h=\delta_{gh}$ the map $\delta:H\to \Lambda(H)^\times$,
$h\mapsto \delta_h$
is a group homomorphism.
\begin{definition}
$\Lambda(H)$ with the above $\bbZ_\ell$-algebra structure is
called the \emph{Iwasawa algebra} of $H$.
\end{definition} 
The following situation will frequently occur in the applications
in this paper. Suppose that 
$$
0\to H\to G\xrightarrow{q}T\to  0
$$
is an exact sequence of profinite groups with $T$ finite discrete. Define for $t\in T$
$$
H\langle t\rangle:=q^{-1}(t)
$$
so that $G=\bigcup_{t\in T}H\langle t\rangle$ and each $H\langle t\rangle$
is an $H$-torsor, i.e., has a simply transitive $H$-action. Then 
$$
\Lambda(G)\isom \bigoplus_{t\in T}\Lambda(H\langle t\rangle)
$$
is a $\Lambda(H)$-module and 
$\Lambda(H\langle t\rangle)$ is a free $\Lambda(H)$-module of rank one. 

\subsection{The moment map}
In this section we consider the profinite group 
$H\isom\bbZ_\ell^d$ and we write
\begin{align}
\begin{split}
H_r:=&H\otimes_{\bbZ_\ell}\bbZ_\ell/\ell^r\bbZ_\ell\isom 
(\bbZ_\ell/\ell^r\bbZ_\ell)^d\\
H_{\bbQ_\ell}:=&H\otimes_{\bbZ_\ell}\bbQ_\ell\isom
\bbQ_\ell^d.
\end{split}
\end{align}
The moment map will be a $\bbZ_\ell$-algebra homomorphism 
$$
\Lambda(H)\to \wh\TSym^\cdot H,
$$
where $\wh\TSym^\cdot H$ is the completion of the $\bbZ_\ell$-algebra
of symmetric tensors with respect to the augmentation ideal.

We start be recalling some facts about the algebra of symmetric
tensors $\TSym^\cdot H$. We remark right away that the right framework
for the moment map is the divided power algebra $\Gamma^\cdot H$,
which in our case is isomorphic to $\TSym^\cdot H$. As we are interested
in the relation with the symmetric algebra in the end, we found it more
intuitive to work with $\TSym^\cdot H$.

The algebra $\TSym^\cdot H$ is graded
$$
\TSym^\cdot H=\bigoplus_{k\ge 0}\TSym^k H
$$
and for each $h\in H$ one has the symmetric 
tensor $h^{[k]}:=h^{\otimes k}\in \TSym^kH$. This gives a divided
power structure on $\TSym^\cdot H$
and one has the formulae 
\begin{align}
\begin{split}
(g+h)^{[k]}&=\sum_{m+n=k}g^{[m]}h^{[n]}\\
h^{[m]}h^{[n]}&=\frac{(m+n)!}{m!n!}h^{[m+n]}.
\end{split}
\end{align}
The map $H\to \TSym^1H$, $h\mapsto h^{[1]}$
is an isomorphism. By the universal property of the symmetric algebra
this induces an algebra homomorphism
\begin{equation}
\Sym^\cdot  H\to \TSym^\cdot H,
\end{equation}
which is an isomorphism after tensoring with $\bbQ_\ell$. 

From the isomorphism $\Gamma^\cdot H\isom\TSym^\cdot H$ it follows 
directly that $\TSym^\cdot H$ is compatible with base change
$$
(\TSym^\cdot H)\otimes_{\bbZ_\ell}\bbZ/\ell^r\bbZ\isom
\TSym^\cdot H_r
$$
and with direct sums 
$\TSym^\cdot (H\oplus H)\isom \TSym^\cdot H\otimes \TSym^\cdot H$.

If $(e_1,\ldots,e_d)$ is a basis of $H$, then 
$$
(e_1^{[n_1]}\cdots e_d^{[n_d]}\mid n_1+\ldots+n_d=k)
$$
is a basis of $\TSym^kH$. 
Note that under the homomorphism $\Sym^kH\to \TSym^kH$ one has
$$
e_1^{n_1}\cdots e_d^{n_d}\mapsto k!e_1^{[n_1]}\cdots e_d^{[n_d]}.
$$
Let $H^\vee:=\Hom_{\bbZ_\ell}(H,\bbZ_\ell)$ be the dual $\bbZ_\ell$-module 
then one has a canonical
isomorphism
$$
\Sym^kH^\vee\isom (\TSym^k H)^\vee.
$$

Let $\TSym^+H:=\bigoplus_{k>0}\TSym^kH$ be the augmentation ideal.
We denote by 
$$
\wh\TSym^\cdot H:=\prolim_n \TSym^\cdot H/(\TSym^+H)^n
$$
the completion of $\TSym^\cdot H$ with respect to the augmentation ideal.
Similarly, we also denote by $\wh\TSym^\cdot H_r$ and 
$\wh\TSym^\cdot H_{\bbQ_\ell}$ the completions with respect to the
augmentation ideal. 
\begin{lemma}\label{TSym-inv-limit} One has 
$$
\wh\TSym^\cdot H\isom \prolim_r \wh\TSym^\cdot H_r.
$$
\end{lemma}
\begin{proof}
As $\TSym^kH$ is a free $\bbZ_\ell$-module, one has 
$$
\TSym^\cdot H/(\TSym^+H)^n\isom 
\prolim_r\TSym^\cdot H_r/(\TSym^+H_r)^n
$$
for all $n\ge 1$. Taking the inverse limit over $n$, the result follows.
\end{proof}
\begin{proposition}\label{mom-def} 
There is a unique homomorphism of $\bbZ_\ell$-algebras
$$
\mom: \Lambda(H)\to \wh\TSym^\cdot H,
$$
which maps $\delta_h\mapsto \sum_{k\ge 0}h^{[k]}$ and is called
the \emph{moment map}. It is the limit
$\mom=\prolim_r\mom_r$ of moment maps at finite level
\begin{align*}
\mom_r:\Lambda_r(H_r)&\to \wh\TSym^\cdot H_r\\
\mu_r&\mapsto \sum_{k\ge 0}(\sum_{h\in H_r}\mu_r(h)h^{[k]}).
\end{align*}
Let $(e_1,\ldots,e_d)$ be a basis of $H$ and $(x_1,\ldots,x_d)$
the dual basis  considered
as $\bbZ_\ell$-valued functions $x_i:H\to \bbZ_\ell$.
In terms of 
measures the moment map is given by 
$$
\mom(\mu)=\sum_{k\ge 0}\left(\sum_{n_1+\ldots+n_d=k}(\int_Hx_1^{n_1}\cdots x_d^{n_d}\mu)
e_1^{[n_1]}\cdots e_d^{[n_d]}\right).
$$
The projection
onto the $k$-th component is denoted by 
$$
\mom^k:\Lambda(H)\to \TSym^kH
$$
and by $\mom_r^k:\Lambda_r(H_r)\to \TSym^k H_r$ respectively.
\end{proposition}
\begin{remark}
The formula for the moment map in terms of measures justifies 
the name. For the application to sheaves of Iwasawa algebras 
it is the formula on finite level which is important.
\end{remark}
\begin{proof}
The map $\mom_r:\Lambda_r(H_r)\to \wh\TSym^\cdot H_r$ in the
proposition is
the algebra homomorphism induced 
by the group homomorphism $h\mapsto \sum_{k\ge 0}h^{[k]}$ and the
universal property of the group algebra $\Lambda_r(H_r)$. Taking
the inverse limit gives $\mom:\Lambda(H)\to \wh\TSym^\cdot H$.

Write $\mu\in\Lambda(H)$ as $\mu=\prolim_r\mu_r$ with $\mu_r\in\Lambda_r(H_r)$.
The dual basis $x_1,\ldots,x_d$ considered as $\bbZ/\ell^r\bbZ$-linear
maps $x_i:H_r\to \bbZ/\ell^r\bbZ$ induce polynomial functions 
$x_i^{n_i}:H_r\to \bbZ/\ell^r\bbZ$ and by definition
$$
\int_{H_r}x_1^{n_1}\cdots x_d^{n_d}\mu_r=\mu_r(x_1^{n_1}\cdots x_d^{n_d})=
\sum_{h\in H_r}\mu_r(h)x_1(h)^{n_1}\cdots x_d(h)^{n_d}.
$$
If we observe that 
$$
\sum_{{n_1+\ldots +n_d=k}}x_1(h)^{n_1}\cdots x_d(h)^{n_d}
{e_1^{[n_1]}\cdots e_d^{[n_d]}}=
(x_1(h)e_1+\cdots +x_d(h)e_d)^{[k]}=h^{[k]}
$$
we get 
$$
\sum_{{n_1+\ldots +n_d=k}}\mu_r(x_1^{n_1}\cdots x_d^{n_d})
{e_1^{[n_1]}\cdots e_d^{[n_d]}}=
\sum_{h\in H_r}\mu_r(h)h^{[k]}=\mom^k_{r}(\mu_r).
$$
This implies that for the measure $\mu=\prolim_r\mu_r\in\Lambda(H)$
we get
$$
\mom^k(\mu)=\prolim_r\mom^k_r(\mu_r)=\prolim_r
\sum_{{n_1+\ldots +n_d=k}}(\int_{H_r}x_1^{n_1}\cdots x_d^{n_d}\mu_r)
{e_1^{[n_1]}\cdots e_d^{[n_d]}},
$$
which implies the desired formula for $\mom(\mu)$.
\end{proof}
Note that the moment map is functorial. If $\varphi:H\to G$ is a group 
homomorphism one has a commutative diagram
$$
\begin{CD}
\Lambda(H)@>\mom >>\wh\TSym^\cdot H\\
@V\phi_!VV@VV\wh\TSym^\cdot(\phi)V\\
\Lambda(G)@>\mom>>\wh \TSym^\cdot G
\end{CD}
$$
It is a fact from classical Iwasawa theory that $\Lambda(H)$ is
isomorphic to a power series ring over $\bbZ_\ell$ in $d$ variables.
In particular, it is a regular local ring. Let 
$$
I(H):=\ker(\Lambda(H)\xrightarrow{\int_H } \bbZ_\ell)
$$
be the augmentation ideal. Then from the regularity of $\Lambda(H)$ it
follows that $I(H)^k/I(H)^{k+1}\isom \Sym^kH$ and the $k$-th moment map
factors
$$
\mom^k:\Lambda(H)\to \Lambda(H)/I(H)^{k+1}\to \TSym^kH.
$$
\begin{lemma}\label{sym-mom-compatibility}
The $k$-th moment map induces 
$$
\Sym^kH\isom I(H)^k/I(H)^{k+1}\hookrightarrow \Lambda(H)/I(H)^{k+1}\xrightarrow{\mom^k}\TSym^kH,
$$
which is just the canonical map. 
\end{lemma}
\begin{proof}
The morphism $\mom^k$ maps
an element $(\delta_{h_1}-1)\cdots(\delta_{h_k}-1)\in I(H)^k$ to the
corresponding product taken in $\TSym^\cdot H$. This implies the result.
\end{proof}

Consider again the exact sequence 
$$
0\to H\to G\xrightarrow{q}T\to 0
$$
of profinite groups with $T$ a finite $N$-torsion group. 
\begin{definition}\label{mom-t-k-defn}
For 
the $H$-torsors $H\langle t\rangle =q^{-1}(t)$  define
$$
\mom_t^k:\Lambda(H\langle t\rangle)\to \TSym^k H
$$
to be the composition  
$$
\mom_t^k:\Lambda(H\langle t\rangle)\xrightarrow{[N]_!}\Lambda(H)
\xrightarrow{\mom^k}\TSym^k H,
$$
where $[N]:G\to G$ is the $N$-multiplication, which factors through $H$. 
\end{definition}
\begin{remark}
This moment map is not independent of the choice of $N$ such
that $t$ is an $N$-torsion point. 
\end{remark}

To remedy this defect consider the composition 
\begin{equation}\label{extended-mom}
\Lambda(H\langle t\rangle)\xrightarrow{\mom^k_t}\TSym^kH\to 
\TSym^k H_{\bbQ_\ell}.
\end{equation}
\begin{definition}
The \emph{modified moment map} 
$$
\wt\mom_t^k:\Lambda(H\langle t\rangle)\to \Sym^k H_{\bbQ_\ell}
$$
is the map \eqref{extended-mom} composed with the inverse of the isomorphism
$\Sym^k H_{\bbQ_\ell}\isom \TSym^k H_{\bbQ_\ell}$ divided by $N^k$, i.e., 
$$
\wt\mom_t^k:=\frac{1}{N^k}\mom_t^k.
$$
\end{definition}
The following lemma is obvious from the definition.
\begin{lemma}
The modified moment map $\wt\mom_t^k$ depends only on $t$ and not
on $N$.
\end{lemma}

%
\subsection{\'Etale sheaves of Iwasawa modules}
%
Consider a projective system of finite \'etale schemes 
$p_r : X_r \to S$ and let
$X := \prolim_rX_r$. 
We denote by $\lambda_r : X_{r+1} \to  X_r$ the finite \'etale maps in the
projective system. We denote by $\sX_r$ the \'etale sheaf associated to 
$X_r$ and
define an \'etale sheaf on $S$ by
\begin{equation}
\Lambda_r(\sX_r):=p_{r*}\bbZ/\ell^r\bbZ.
\end{equation}
The trace map with respect to $\lambda_r$ 
induces a morphism of sheaves $\lambda_{r*}\bbZ/\ell^{r+1}\bbZ\to \bbZ/\ell^{r+1}\bbZ$
which gives rise to 
$$
\Lambda_{r+1}(\sX_{r+1})=
p_{r*}\lambda_{r*}\bbZ/\ell^{r+1}\bbZ\to
p_{r*}\bbZ/\ell^{r+1}\bbZ\to p_{r*}\bbZ/\ell^{r}\bbZ=\Lambda_r(\sX_r),
$$
where the last map is reduction modulo $\ell^r$.
\begin{definition}
Define an inverse system of \'etale sheaves on $S$ by 
$$
\Lambda(\sX):=(\Lambda_r(\sX_r))_{r\ge 1},
$$
with the above transition maps.
\end{definition}
\begin{remark}
Note that $\Lambda(\sX)$ is not an $\ell$-adic sheaf in general.
\end{remark}
This construction is functorial in the sense that for a morphism of inverse 
systems $(f_r : X_r \to Y_r)_{r\ge 1}$ the trace map induces
\begin{equation}
f_{r!}:\Lambda_r(\sX_r)\to \Lambda_r(\sY_r)
\end{equation}
and hence a map $f_!:\Lambda(\sX)\to \Lambda(\sY)$.

We want to explain in which sense $\Lambda(\sX)$ is a sheafification of
the space of  measures $\Lambda(X)$.

Let us choose a geometric point $\ol s : \Spec \ol K \to  S $
and let $\sX_{r,\ol s}$ be the stalk of $\sX_r$ at $\ol s$.
We consider $\sX_{r,\ol s}$ as a finite set with a continuous
Galois action. Immediately from the definitions we have:
\begin{lemma}\label{relation-to-Iwasawa-modules}
The stalk of $\Lambda_r(\sX_r)$ at $\ol s$ is 
$$
\Lambda_r(\sX_r)_{\ol s}\isom \Lambda_r(\sX_{r,\ol s}).
$$
In particular, if we define 
$\Lambda(\sX)_{\ol s}:=\prolim_r\Lambda_r(\sX_r)_{\ol s}$ and
$\sX_{\ol s}:=\prolim_r\sX_{r,\ol s}$ we get
$$
\Lambda(\sX)_{\ol s}\isom \Lambda(\sX_{\ol s}),
$$
which is the space of  measures on $\sX_{\ol s}$ with a Galois action.
\end{lemma}
In the case where each $X_r=H_r\xrightarrow{p_r}S$ a finite \'etale group scheme over $S$,
so that $H:=\prolim_r H_r$ is a pro-\'etale group scheme, the
sheaves $\Lambda_r(\sH_r)$ become sheaves of $\bbZ/\ell^r\bbZ$-algebras.
In fact one has
$$
(p_r\times p_r)_*\bbZ/\ell^r\bbZ\isom \Lambda_r(\sH_r)\otimes \Lambda_r(\sH_r)
$$
and the group multiplication induces a ring structure on $\Lambda_r(\sH_r)$.
%
\subsection{The case of torsors}\label{torsor-section}
%
The following situation will occur very frequently in this paper.
Suppose we have an inverse system of finite \'etale group schemes on $S$
\begin{equation}\label{torsor-r-sequence}
0\to H_r\to G_r\xrightarrow{q_r} T\to 0
\end{equation}
where $T=T_r$ for all $r$ is an $N$-torsion group. For each section
$t:S\to T$ we define an $H_r$-torsor $H_r\langle t\rangle$ by the
cartesian diagram
\begin{equation}\label{torsor-r-def}
\begin{CD}
H_r\langle t\rangle@>>>G_r\\
@Vp_{r,t} VV@VVV\\
S@>t>>T.
\end{CD}
\end{equation}
Denote by $H:=\prolim_r H_r$, $G:=\prolim_r G_r$ and 
$H\langle t\rangle:=\prolim_r H_r\langle t \rangle$ the associate
pro-\'etale group schemes and by $\sH_r$, $\sG_r$, $\sH_r\langle t\rangle$
and $\sH$, $\sG$, $\sH\langle t\rangle$ the associated sheaves.
In particular one has an exact sequence
\begin{equation}\label{torsor-sequence}
0\to H\to G\xrightarrow{q}T\to 0
\end{equation}
and a cartesian diagram 
\begin{equation}
\begin{CD}
H\langle t\rangle@>>>G\\
@Vp_t VV@VVV\\
S@>t>>T.
\end{CD}
\end{equation}
Each $\Lambda_r(\sH_r\langle t\rangle)$ is a $\Lambda_r(\sH_r)$-module
of rank one and consequently the same is true for 
the $\Lambda(\sH)$-module $\Lambda(\sH\langle t\rangle)$.

The sheaves $\Lambda(\sH\langle t\rangle)$ are sheaves of Iwasawa 
modules under the sheaves of Iwasawa algebras $\Lambda(\sH)$. 

%
\subsection{The sheafified  moment map}
%

In 
this section we describe a sheaf version of the moment maps 
from Proposition \ref{mom-def}.

Let $p_r:H_r\to S$ be a finite \'etale group scheme which is \'etale locally
of the form 
$$
H_r\isom (\bbZ/\ell^r\bbZ)^d
$$ 
for $d\ge 1$.
As in \eqref{torsor-r-def} we consider $H_r$-torsors 
$p_{r,t}:H_r\langle t\rangle\to S$
associated to an exact sequence
$$
0\to H_r\to G_r\to T\to 0
$$
and to an $N$-torsion section $t$ of $T$. As $T$ is an $N$-torsion group
by assumption, the $N$-multiplication map $[N]:G_r\to G_r$ factors through
$H_r$ and we get a map of schemes
\begin{equation}
\tau_{r,t}:H_r\langle t\rangle\hookrightarrow G_r\xrightarrow{[N]} H_r.
\end{equation}
We interpret this as a section 
$\tau_{r,t}\in H^0(H_r\langle t\rangle,p_{r,t}^*\sH_r)$.
\begin{definition}\label{tau-r-t-def}
We let 
\begin{equation*}
\tau_{r,t}^{[k]}\in H^0(H_r\langle t\rangle,p_{r,t}^*\TSym^k\sH_r)
\end{equation*}
be the $k$-th tensor power $\tau_{r,t}$. This will also be viewed
as a map of sheaves 
\begin{equation*}
\tau_{r,t}^{[k]}:\bbZ/\ell^r\bbZ\to p_{r,t*}p_{r,t}^*\TSym^k\sH_r.
\end{equation*}
\end{definition}

Recall that for sheaves $\sF, \sG$ on $H_r\langle t\rangle$ 
one has the morphism
(given by the projection formula and adjunction)
\begin{equation}\label{cup-with-cpt-supp}
p_{r,t,!}\sF\otimes p_{r,t*}\sG\isom p_{r,t,!}(\sF\otimes p_{r,t}^*p_{r,t*}\sG)\to p_{r,t,!}(\sF\otimes \sG)
\end{equation}
and that $p_{r,t,!}=p_{r,t*}$ as $p_{r,t}$ is finite. 
\begin{definition}\label{sheaf-mom-def} 
The sheafified moment map
$$
\mom^k_{r,t}:\Lambda_r(\sH_r\langle t\rangle)\to \TSym^k\sH_r
$$
is the composition ($p:=p_{r,t}$)
\begin{multline*}
p_{*}\bbZ/\ell^r\bbZ\xrightarrow{\id\otimes \tau_{r,t}^{[k]}}
p_{*}\bbZ/\ell^r\bbZ\otimes p_{*}p^*\TSym^k\sH_r\xrightarrow{\eqref{cup-with-cpt-supp}}\\
p_*(\bbZ/\ell^r\bbZ\otimes p^*\TSym^k\sH_r)\isom p_{*} p^*\TSym^k\sH_r
\xrightarrow{\tr}\TSym^k\sH_r,
\end{multline*}
where $\tr$ is the trace map with respect to $p$.
\end{definition}
From the definition it follows that the moment maps 
$\mom_{r,t}^k$ are compatible with respect to the trace map for varying $r$.
\begin{definition}\label{wtmom-def}
Let 
\begin{equation*}
\mom^k_t:\Lambda(\sH\langle t\rangle)\to \TSym^k\sH
\end{equation*}
be the inverse limit of $\mom_{r,t}^k$. We also let 
$$
\wt{\mom}_t^k:=\frac{1}{N^k}\mom_t^k:
H^1(S,\Lambda(\sH\langle t\rangle)(1))\to
H^1(S,\Sym^k\sH_{\bbQ_\ell}(1))
$$
be the composition of the map induced by $\frac{1}{N^k}\mom^k_t$ 
in cohomology with the inverse of the canonical isomorphism
$$
H^1(S,\Sym^k\sH_{\bbQ_\ell}(1))\isom H^1(S,\TSym^k\sH_{\bbQ_\ell}(1)).
$$
\end{definition} 
On stalks the sheafified moment map coincides with the 
one defined in Definition  \ref{mom-t-k-defn}:
\begin{lemma}
Let $\ol s$ be a geometric point of $S$, then the stalk of the 
moment map
$$
(\mom^k_{r,t})_{\ol s}:\Lambda_r(\sH_r\langle t\rangle)_{\ol s}\to \TSym^k\sH_{r,\ol s}
$$
coincides with the moment map $\mom_{r,t}^k$ defined in 
Definition  \ref{mom-t-k-defn}.
\end{lemma}
\begin{proof}
We have $\Lambda_r(\sH_r\langle t\rangle)_{\ol s}=\Lambda_r(\sH_r\langle t\rangle_{\ol s})$
and we let
$$
\mu_r=\sum_{\ol x\in \sH_{r}\langle t\rangle_{\ol s}}m_{\ol x}\delta_{\ol x}
$$
be an element in $\Lambda_r(\sH_r\langle t\rangle_{\ol s})$. We
identify ($p:=p_{r,t}$)
$$
p_*p^*\TSym^k\sH_{r,\ol s}\isom \Lambda_r(\sH_{r,\ol s})\otimes \TSym^k\sH_{r,\ol s}
$$
so that 
$$
(p_*\bbZ/\ell^r\bbZ\otimes p_*p^*\TSym^k\sH_r)_{\ol s}\isom
\Lambda_r(\sH_{r,\ol s})\otimes  \Lambda_r(\sH_{r,\ol s})\otimes \TSym^k\sH_{r,\ol s}.
$$
With this identification the image of $\mu_r$ under 
$\id\otimes\tau^{[k]}_{r,t} $ is given by
\begin{equation}\label{element}
(\sum_{\ol x\in \sH_{r}\langle t\rangle_{\ol s}}m_{\ol x}\delta_{\ol x})\otimes 
(\sum_{\ol y\in  \sH_{r}\langle t\rangle_{\ol s}}\delta_{\ol y}\otimes 
\tau_{r,t}^{[k]}(\ol y)).
\end{equation}
The homomorphism 
$$
\Lambda_r(\sH_{r,\ol s})\otimes  \Lambda_r(\sH_{r,\ol s})\otimes \TSym^k\sH_{r,\ol s}\xrightarrow{\eqref{cup-with-cpt-supp}}
\Lambda_r(\sH_{r,\ol s})\otimes \TSym^k\sH_{r,\ol s}
$$
maps the element in \eqref{element} to 
$$
\sum_{\ol x\in  \sH_{r}\langle t\rangle_{\ol s}}m_{\ol x}\delta_{\ol x}\otimes 
\tau_{r,t}^{[k]}(\ol x)=
\sum_{\ol x\in  \sH_{r}\langle t\rangle_{\ol s}}\mu_r(\ol x)
\delta_{\ol x}\otimes
\tau_{r,t}^{[k]}(\ol x).
$$
and the trace of this is 
$$
\sum_{\ol x\in  \sH_{r}\langle t\rangle_{\ol s}}\mu_r(\ol x)
\tau_{r,t}^{[k]}(\ol x)=
\sum_{\ol x\in  \sH_{r}\langle t\rangle_{\ol s}}\mu_r(\ol x)([N]\ol x)^{[k]}
$$
which is the desired formula.
\end{proof}

%
\subsection{Soul\'e's twisting construction and the moment map}
%
Let us consider the situation in \eqref{torsor-sequence}
$$
0\to H\to G\xrightarrow{q}T\to 0
$$
and recall the inverse system of $H_r$-torsors $H_r\langle t\rangle$. 
Denote by $\lambda_r:H_{r+1}\langle t\rangle
\to H_r\langle t\rangle$ the transition maps and by 
$p_{r,t}: H_r\langle t\rangle\to S$ the structure map.
\begin{definition}\label{norm-comp-def}
A \emph{norm-compatible function} $\theta=(\theta_r)_{r\ge 1}$ on 
$H\langle t\rangle=( H_r\langle t\rangle)_{r\ge 1}$ is an inverse
system of global invertible sections 
$$
\theta_r\in \Gm(H_r\langle t\rangle)
$$
such that $\lambda_{r*}(\theta_{r+1})=\theta_{r}$, where 
$\lambda_{r*}$ is the norm map with respect to $\lambda_r$. 
\end{definition}
\begin{definition}\label{Kummer-def}
The \emph{Kummer map}
\begin{equation*}
\partial_r:\Gm(H_r\langle t\rangle)\to H^1(H_r\langle t\rangle, \mu_{\ell^r}),
\end{equation*}
is the boundary map for the exact sequence
$$
0\to \mu_{\ell^r}\to \Gm\xrightarrow{[\ell^r]}\Gm\to 0.
$$
\end{definition}
Recall the section 
$\tau_{r,t}^{[k]}\in H^0(H_r\langle t\rangle,p_{r,t}^*\TSym^k\sH_r)$ 
from Definition \ref{tau-r-t-def}.
Soul\'e's twisting construction is now as follows. 
\begin{definition}
Let 
\begin{equation*}
s(r,k,t):=p_{r,t*}(\partial_r(\theta_r)\cup
\tau_{r,t}^{[k]})\in H^1(S,\TSym^k\sH_r(1)),
\end{equation*}
where we have written $\TSym^k\sH_r(1):=\TSym^k\sH_r\otimes \mu_{\ell^r}$
as usual.
\end{definition}

Recall from Lemma \ref{TSym-inv-limit} that
$\prolim_r\TSym^k\sH_r=\TSym^k\sH$ and denote by 
$\red_r:\TSym^k\sH_{r+1}\to \TSym^k\sH_{r}$ the reduction modulo $\ell^r$.
\begin{proposition}[Soul\'e]\label{Soule-twisting}
Under the
transition maps
$$
\red_r:H^1(S,\TSym^k\sH_{r+1}(1))\to H^1(S,\TSym^k\sH_r(1))
$$
one has $\red_r(s(r+1,k,t))=s(r,k,t)$.
In particular, one gets an element
$$
s(k,t):=\prolim_r s(r,k,t)
\in  H^1(S,\TSym^k\sH(1)).
$$
We refer to this construction as \emph{Soul\'e's twisting construction}.
\end{proposition}
\begin{remark}
In general Soul\'e's twisting construction allows also to construct
elements in other Galois representations and it depends on the 
choice of elements in the Galois representation. Here we have fixed 
the tautological sections $\tau_{r,t}^{[k]}$ of $\TSym^k\sH_r$ to define
this twist. In \cite{Kings-Hornbostel} one can find more general twisting
constructions.
\end{remark}
\begin{proof} By abuse of notation we also denote by $\red_r$
any map on cohomology which reduces the coefficient module modulo $\ell^r$.
Then one has $\red_r\verk \lambda_{r*}=\lambda_{r*}\verk \red_r$.
We have $\red_r(\tau_{r+1,t}^{[k]})=\lambda_r^*(\tau_{r,t}^{[k]})$ 
and by assumption
$\red_r\verk \lambda_{r*}(\theta_{r+1})=\lambda_{r*}\verk\red_{r*}(\theta_{r+1})=\theta_r$.
Then 
\begin{align*}
\red_r(s(r+1,k,t))&= \red_r \verk p_{r+1,t*}(\partial_{r+1}(\theta_{r+1})\cup
\tau_{r+1,t}^{[k]})\\
&= p_{r+1,t*}(\red_r(\partial_{r+1}(\theta_{r+1}))\cup 
\red_r(\tau_{r+1,t}^{[k]}))\\
&=p_{r,t*}\verk \lambda_{r*}(\red_r(\partial_{r+1}(\theta_{r+1}))\cup 
\lambda_r^*(\tau_{r,t}^{[k]}))\\
&=p_{r,t*}(\lambda_{r*}\verk\red_r(\partial_{r+1}(\theta_{r+1}))\cup
\tau_{r,t}^{[k]})\\
&=p_{r,t*}(\partial_r(\theta_r)\cup \tau_{r,t}^{[k]})\\
&=s(r,k,t).
\end{align*}
\end{proof}
The following identification is fundamental for the whole paper.
\begin{lemma}
Let  $p_{r,t}:H_r\langle t\rangle\to S$ be the $H_r$-torsor as above,
then one has a canonical isomorphism
$$
H^i(H_r\langle t\rangle, \mu_{\ell^r})\isom
H^i(S,\Lambda_r(\sH_r\langle t\rangle)(1)).
$$
\end{lemma}
\begin{proof}
As $p_{r,t}$ is finite this follows from the Leray spectral sequence.
\end{proof}
With this identification we can rewrite the Kummer map and one gets
a commutative diagram
\begin{equation}
\begin{CD}
\Gm(H_{r+1}\langle t\rangle)@>\partial_{r+1}>> H^1(S, \Lambda_{r+1}(\sH_{r+1}\langle t\rangle)(1))\\
@V\lambda_{r*}VV@VV\lambda_{r*}V\\
\Gm(H_{r}\langle t\rangle)@>\partial_{r}>> 
H^1(S, \Lambda_{r}(\sH_{r}\langle t\rangle)(1)),
\end{CD}
\end{equation}
where the $\lambda_{r*}$ on the right hand side is induced by the
trace map $\lambda_{r!}:\Lambda_{r+1}(\sH_{r+1}\langle t\rangle)\to 
\Lambda_{r}(\sH_{r}\langle t\rangle)$. This diagram allows to 
consider the inverse limit of the $\partial_r(\theta_r)$:
\begin{definition}
The norm-compatible functions $\theta=(\theta_r)_{r\ge 1}$ define an
element 
$$
\cS^{\langle t\rangle}:=\prolim_r\partial_r(\theta_r)\in H^1(S,\Lambda(\sH\langle t\rangle)(1))
=\prolim_r   H^1(S,\Lambda_r(\sH_r\langle t\rangle)(1)).
$$
We also let $\cS_r^{\langle t\rangle}:=\partial_r(\theta_r)$.
\end{definition}
With this preliminaries we can finally explain the crucial relation between
the moment map and Soul\'e's twisting  construction.
\begin{proposition}\label{mom-and-twist} The homomorphism 
$$
\mom^k_{r,t}: H^1(S, \Lambda_r(\sH_r\langle t\rangle)(1))
\to H^1(S, \TSym^k\sH_r(1))
$$
induced by the moment map $\mom^k_{r,t}$ coincides with the composition
\begin{multline*}
H^1(S, \Lambda_r(\sH_r\langle t\rangle)(1))\isom
H^1(H_r\langle t\rangle, \bbZ/\ell^r\bbZ(1))
\xrightarrow{\cup \tau_r^{[k]}}\\
H^1(H_r\langle t\rangle, p_{r,t}^*\TSym^k\sH_r(1))\xrightarrow{p_{r,t*}}H^1(S, \TSym^k\sH_r(1)).
\end{multline*}
In particular, one has $\mom_{r,t}^k(\cS_r^{\langle t\rangle})=s(r,k,t)$ 
and in 
the limit
$$
\mom_t^k(\cS^{\langle t\rangle})=s(k,t).
$$
\end{proposition}
\begin{proof} Let $p:=p_{r,t}$ then the result follows 
from the commutative diagram
$${\small
\xymatrix{H^1(H_r\langle t\rangle, \mu_{\ell^r})\times 
H^0(H_r\langle t\rangle,  p^*\TSym^k\sH_r)
\ar[r]^{\cup}\ar[dd]^\isom&
H^1(H_r\langle t\rangle, p^*\TSym^k\sH_r(1))\ar[d]^\isom\\
&H^1(S,p_*(\bbZ/\ell^r\bbZ\otimes p^*\TSym^k\sH_r(1)))\\
H^1(S, p_*\mu_{\ell^r})\times H^0(S,p_*p^*\TSym^k\sH_r)
\ar[r]^{\cup}&
H^1(S, p_*\bbZ/\ell^r\bbZ\otimes p_*p^*\TSym^k\sH_r(1))\ar[u]_{\eqref{cup-with-cpt-supp}}
}}
$$
\end{proof}

\section{Three examples}

\subsection{The Bernoulli measure and its moments}
\label{Zell-ex}
Let $N>1$ and $t\in \bbZ/N\bbZ$ and consider for each $r\ge 0$ the
exact sequence
$$
0\to \bbZ/\ell^r\bbZ\to \bbZ/\ell^rN\bbZ\xrightarrow{q_r}\bbZ/N\bbZ\to 0,
$$
where $q_r$ is reduction modulo $N$. We let 
$Z_r:=\bbZ/\ell^r\bbZ$ and $Z:=\bbZ_\ell$. In the notation of \eqref{torsor-sequence}
we have $H_r= \bbZ/\ell^r\bbZ$, $G_r= \bbZ/\ell^rN\bbZ$ and
$T=\bbZ/N\bbZ$.
We define 
$$
Z_r\langle t\rangle:=q_r^{-1}(t)=
\{x\in \bbZ/\ell^rN\bbZ\mid x\equiv t \mod{N} \}
$$
so that $Z_r\langle t\rangle =H_r\langle t\rangle$ in the notation
of \eqref{torsor-sequence}. We denote by 
$$
Z\langle t\rangle :=\prolim_r Z_r\langle t\rangle
$$ 
the inverse limit.
As before, each $Z_r\langle t\rangle$ is an $Z_r$-torsor.
Recall that $Z\langle 0\rangle=Z=\bbZ_\ell$ and that 
$\Lambda(Z\langle t\rangle)$
is a free rank one $\Lambda(Z)$-module.

Let us define the Bernoulli measure in $\Lambda(Z\langle t\rangle)$. We 
choose, as usual, an auxiliary $c\in\bbZ$ with $(c,\ell N)=1$
to make the Bernoulli distribution integral (for the properties
of the Bernoulli numbers we refer to 
\cite[Ch. 2, \S 2]{Lang-cyclotomic-fields}).
\begin{definition}\label{Bernoulli-c-r-def}
Denote by $B_k(x)$ the $k$-th Bernoulli polynomial. The map  
\begin{align}
\begin{split}
B_{2,c,r}^{\langle t\rangle}:Z_r\langle t\rangle&\to \bbZ/\ell^r\bbZ\\
x&\mapsto \frac{\ell^rN}{2}(c^2B_2(\{\frac{x}{\ell^rN}\})-
B_2(\{\frac{cx}{\ell^rN}\})),
\end{split}
\end{align}
where for an element $x\in \bbR/\bbZ$ we write $\{ x\}$ for its 
representative in $[0,1[$, defines an element
$$
B_{2,c,r}^{\langle t\rangle}\in\Lambda_r(Z_r\langle t\rangle).
$$
\end{definition}
By the distribution property of the Bernoulli polynomials 
the $B_{2,c,r}$ are compatible under the trace map
$\Lambda_{r+1}(Z_{r+1}\langle t\rangle)\to \Lambda_r(Z_r\langle t\rangle)$
and give rise to a measure 
\begin{equation}\label{Bernoulli-measure-def}
B_{2,c}^{\langle t\rangle}:=\prolim_r B_{2,c,r}^{\langle t\rangle}\in \Lambda(Z\langle t\rangle). 
\end{equation}
We want to compute the moments of the Bernoulli measure.
Choose $e=1\in \bbZ_\ell$ as a basis and let $x=\id:\bbZ_\ell\to \bbZ_\ell$
be the dual basis. By standard congruences for 
Bernoulli polynomials (see e.g. \cite[Theorem 2.1]{Lang-cyclotomic-fields})
we have
\begin{align}\label{mom-of-Bernoulli}
\begin{split}
\mom^k_t (B_{2,c}^{\langle t\rangle})&=\int_{Z\langle t\rangle}x^kdB_{2,c}^{\langle t\rangle}\\
&=\frac{N^{k+1}}{c^k(k+2)}(c^{k+2}B_{k+2}(\{\frac{t}{N}\})-
B_{k+2}(\{\frac{ct}{N}\})).
\end{split}
\end{align}
Note that if $c\equiv 1\mod{N}$ we get 
\begin{equation}
\mom^k_t (B_{2,c}^{\langle t\rangle})=\frac{N^{k+1}(c^{k+2}-1)}{c^k(k+2)}B_{k+2}(\{\frac{t}{N}\}).
\end{equation}

%
\subsection{Modified cyclotomic Soul\'e-Deligne elements} \label{cyc-soule}
%
We review the cyclotomic elements defined by Soul\'e \cite{Soule-p-adic-regulators}
and Deligne \cite{Deligne-groupe-fondamental} from
our perspective. In the literature two kinds of Soul\'e-Deligne elements
are in use. There are the ones used in Iwasawa theory 
obtained by using the norm of the field extension 
$\bbQ(\mu_{\ell^rN})/\bbQ(\mu_N)$ and the ones which
come from motivic cohomology via the regulator map. These are obtained
by the trace map from $[\ell^r]:\mu_{\ell^rN}\to \mu_N$. The relation between
these two elements is essentially an Euler factor (see the discussion 
in \cite{Huber-pune}). We treat here only the later elements originating from
motivic cohomology.

Let $N> 1$ and consider the exact sequence of
finite \'etale group schemes over a base $S$
$$
0\to \mu_{\ell^r}\to \mu_{\ell^rN}\to \mu_N\to 0.
$$
With the notations in \eqref{torsor-r-sequence} we have
$H_r=\mu_{\ell^r}$, $G_r=\mu_{\ell^rN}$ and $T=\mu_N$. 
For each $1\neq \alpha\in \mu_N(S)$ we define as in \eqref{torsor-r-def}
the $\mu_{\ell^r}$-torsor 
$\mu_{\ell^r}\langle \alpha \rangle$ by
the cartesian diagram
\begin{equation}
\begin{CD}
\mu_{\ell^r}\langle \alpha \rangle
@>>>\mu_{\ell^rN}\\
@Vp_{r,\alpha} VV@VVV\\
S@>\alpha>>\mu_N.
\end{CD}
\end{equation}
The inverse limit of these $\mu_{\ell^r}$-torsors is denoted by
$$
\sT\langle \alpha \rangle:=\prolim_r \mu_{\ell^r}\langle \alpha \rangle
$$
and we use the same notation for the associated sheaves.

On $\Gm\setminus\{1\}$ we have the invertible function
\begin{align*}
\Xi:\Gm\setminus\{1\}&\to \Gm\\
z&\mapsto 1-z
\end{align*}
and it is well-known that $\Xi$ is norm-compatible: if $[\ell^r]$
denotes the $\ell^r$-multiplication on $\Gm$ one has
\begin{equation}
[\ell^r]_*(\Xi)=\Xi.
\end{equation}
Thus we can restrict $\Xi$ to $\mu_{\ell^r}\langle\alpha \rangle$
to get norm-compatible functions $\theta_r$ in the notation of 
Definition \ref{norm-comp-def}.
\begin{definition}
We let 
$$
\cC\cS_r^{\langle\alpha\rangle}:=\partial_r(\Xi)\in 
H^1(S,\Lambda_r(\mu_{\ell^r}\langle\alpha \rangle)(1))
$$
and define 
$$
\cC\cS^{\langle\alpha\rangle}:=\prolim_r
\cC\cS_r^{\langle\alpha\rangle}\in 
H^1(S,\Lambda(\sT\langle\alpha \rangle)(1))
$$
\end{definition}
The section $\tau_{r,\alpha}$ from Definition \ref{tau-r-t-def} is the map
$$
\tau_{r,\alpha}:\mu_{\ell^r}\langle\alpha\rangle\hookrightarrow\mu_{\ell^rN}\xrightarrow{[N]}
\mu_{\ell^r}
$$
and its $k$-th tensor
power gives
\begin{equation}
\tau^{[k]}_{r,\alpha}\in H^0(\mu_{\ell^r}\langle\alpha\rangle,\bbZ/\ell^r\bbZ(k)).
\end{equation}
\begin{definition}
Let $1\neq \alpha\in \mu_N(S)$.
We denote by 
\begin{equation*}
\wt c_{k+1,r}(\alpha):=
p_{r,\alpha *}(\partial_r(\Xi)\cup \tau^{[k]}_{r,\alpha})\in H^1(S,\bbZ/\ell^r\bbZ(k+1))
\end{equation*}
the element $s(r,k,\alpha)$ obtained by Soul\'e's twisting construction.
\end{definition}

Note that for $S:=\Spec \bbQ(\mu_{\ell^rN})$ the section
$\tau^{[k]}_{r,\alpha}$ is given by
$\beta\mapsto (\beta^N)^{\otimes k}$
for $\beta\in \mu_{\ell^r}\langle\alpha \rangle(S)$. Moreover, one has 
$$
\mu_{\ell^r}\langle \alpha\rangle(S)=
\{\beta\in \mu_{\ell^rN}(S)\mid \beta^{\ell^r}
=\alpha\}.
$$
It follows that over $S:=\Spec \bbQ(\mu_{\ell^rN})$ the element
$\wt c_{k+1,r}(\alpha)$
is given explicitly by
\begin{equation}
\wt c_{k+1,r}(\alpha)=\sum_{\beta\in \mu_{\ell^r}\langle \alpha\rangle(S)}
\partial_r(1-\beta)\cup(\beta^{N})^{\otimes k}.
\end{equation}

\begin{definition}\label{cyc-soule-def}
For $1\neq\alpha\in \mu_N(S)$ and  $k\ge 1$ 
the \emph{modified cyclotomic Soul\'e-Deligne element}
is
$$
\wt c_{k+1}(\alpha):=\prolim_r \wt c_{k+1,r}(\alpha)
\in H^1(S,\bbZ_{\ell}(k+1)).
$$
Moreover, for a function $\psi:\mu_N(S)\to \bbZ_\ell$ we let
$$
\wt c_{k+1}(\psi):=
\sum_{\alpha\in\mu_N(S) }\psi(\alpha)\wt c_{k+1}(\alpha).
$$
\end{definition}
From the general result in Proposition \ref{mom-and-twist} we
get the following relation between $\cC\cS^{\langle \alpha\rangle}$
and $\wt c_{k+1}(\alpha)$ under the moment map
$$
\mom_\alpha^k:H^1(S,\Lambda(\sT\langle\alpha \rangle)(1))\to
H^1(S,\bbZ_{\ell}(k+1)).
$$
\begin{proposition}\label{mom-of-CS}
In $H^1(S,\bbZ/{\ell^r}\bbZ(k+1))$ one has
$$
\mom_{r,\alpha}^k (\cC\cS_r^{\langle \alpha\rangle})=
\wt c_{k+1,r}(\alpha)
$$
and in the limit
$$
\mom_{\alpha}^k (\cC\cS^{\langle \alpha\rangle})=
\wt c_{k+1}(\alpha)
$$
\end{proposition}
For later use we need a variant of $\cC\cS_r^{\langle \alpha\rangle}$.
Fix an integer $c>1$ which is prime to $\ell N$. The
function 
\begin{equation}\label{Xi-def}
_c\Xi:\Gm\setminus \mu_c\to \Gm
\end{equation}
defined by $z\mapsto \frac{(1-z)^{c^2}}{1-z^c}$ is 
norm-compatible and one has $_c\Xi=\Xi^{c^2}([c]^*\Xi)^{-1}$.
Note that the $[c]$-multiplication maps 
$$
[c]:\mu_{\ell^r}\langle \alpha\rangle\isom 
\mu_{\ell^r}\langle \alpha^c\rangle.
$$
\begin{definition}\label{CS-def}
Let 
$$
\cC\cS_{c,r}^{\langle\alpha\rangle}:=
\partial_r({_c\Xi})=c^2\cC\cS_r^{\langle\alpha\rangle}-
[c]^*\cC\cS_r^{\langle\alpha^c\rangle}
$$
in $H^1(S,\Lambda_r(\mu_{\ell^r}\langle\alpha \rangle)(1))$ and define
$$
\cC\cS_{c}^{\langle\alpha\rangle}:=\prolim_r
\cC\cS_{c,r}^{\langle\alpha\rangle}\in
H^1(S,\Lambda(\sT\langle\alpha \rangle)(1)).
$$
\end{definition}
We compute the moments of $\cC\cS_{c}^{\langle\alpha\rangle}$.
\begin{proposition}\label{mom-of-S}
Let $k\ge 1$  then 
$$
\mom^k_\alpha(\cC\cS_c^{\langle\alpha\rangle})=
c^2\wt c_{k+1}(\alpha)-c^{-k}\wt c_{k+1}(\alpha^c).
$$
In particular, for $c\equiv 1\mod{N}$ one has 
$$
\mom^k_\alpha(\cC\cS_c^{\langle\alpha\rangle})=\frac{c^{k+2}-1}{c^k}
\wt c_{k+1}(\alpha).
$$
\end{proposition}
\begin{proof}
There is a commutative diagram
$$
\xymatrix{\Lambda_r(\mu_{\ell^r}\langle\alpha\rangle)\ar[r]^{[c]_!}
\ar[d]_{\mom^k_{r,\alpha}}&\Lambda_r(\mu_{\ell^r}\langle\alpha^c\rangle)
\ar[d]^{\mom^k_{r,\alpha^c}}\\
\bbZ/\ell^r\bbZ(k)\ar[r]^{[c]^k}&
\bbZ/\ell^r\bbZ(k)
}
$$
and  as $[c]:\mu_{\ell^r}\langle \alpha\rangle\isom 
\mu_{\ell^r}\langle \alpha^c\rangle$ is an isomorphism, one has 
$[c]_*[c]^*=\id$ and one computes
$$
c^k\mom^k_{r,\alpha}([c]^*\cC\cS_r^{\langle\alpha^c\rangle})
=\mom^k_{r,\alpha^c}([c]_*[c]^*\cC\cS_r^{\langle\alpha^c\rangle})=
\mom^k_{r,\alpha^c}(\cC\cS_r^{\langle\alpha^c\rangle})=\wt c_{k+1,r}(\alpha^c)
$$
with Proposition \ref{mom-of-CS} and the result follows.
\end{proof}
%
\subsection{Elliptic Soul\'e elements}
%
We use the theory of norm-compatible 
elliptic units as developed by Kato. 

First we fix an analytic uniformization of $Y(N)(\bbC)$ which
is the same as in \cite{Kato-p-adic}.  Note that
$\sigma\in\GL_2(\bbZ/N\bbZ)$ acts from the left on $Y(N)$ by 
$\sigma\alpha(v):=\alpha(v\sigma)$ for all $v\in (\bbZ/N\bbZ)^2$.
Let $\bbH:=\{\tau\in \bbC\mid \Im \tau >0\}$ be the upper half plane, then
one has an analytic uniformization
\begin{align}
\begin{split}
\nu:(\bbZ/N\bbZ)^\times \times (\Gamma(N)\backslash \bbH)&\xrightarrow{\isom}Y(N)(\bbC)\\
(a,\tau)&\mapsto (\bbC/(\bbZ\tau+\bbZ),\alpha),
\end{split}
\end{align}
where $\Gamma(N):=\ker(\SL_2(\bbZ)\to \SL_2(\bbZ/N\bbZ))$ and $\alpha$
is the level structure given by $(v_1,v_2)\mapsto \frac{av_1\tau+v_2}{N}$.

Recall the main theorem from \cite{Kato-p-adic}.
\begin{theorem}[Kato \cite{Kato-p-adic} 1.10.]\label{modular-units}
Let $\cE$ be an elliptic curve over a scheme $S$ and $c$ be an integer prime
to $6$, then there exists a unit $_c\vartheta_{\cE}\in \cO({\cE}\setminus {\cE}[c])^\times$
such that 
\begin{enumerate}
\item $\div _c\vartheta_{\cE}=c^2(0)- {\cE}[c]$
\item $[d]_*{ _c\vartheta_{\cE}}={_c\vartheta_{\cE}}$ for all $d$ prime to $c$
\item If $\varphi:{\cE}\to {\cE}'$ is an isogeny of elliptic curves over $S$ with
$\deg \varphi$ prime to $c$, then 
$$
\varphi_*(_c\vartheta_{\cE})={_c\vartheta_{{\cE}'}}.
$$
\item For $\tau\in \bbH$ and $z\in \bbC\setminus c^{-1}(\bbZ\tau +\bbZ)$
let $_c\vartheta(\tau,z)$ be the value at $z$ of $_c\vartheta_{\cE}$ for the elliptic curve 
${\cE}=\bbC/(\bbZ\tau +\bbZ)$ over $\bbC$. Then
$$
 _c\vartheta(\tau,z)=q_\tau^{(c^2-1)/12}(-q_z)^{(c-c^2)/2}
 \frac{(1-q_z)^{c^2}}{1-q_z^c}\wt\gamma_{q_\tau}(q_z)^{c^2}
 \wt\gamma_{q_\tau}(q_z^c)^{-1}
$$
where $q_\tau:=e^{2\pi i\tau}$, $q_z:=e^{2\pi i z}$ and
$$
\wt\gamma_{q_\tau}(t):=\prod_{n\ge 1}(1-q_\tau^nq_z)
\prod_{n\ge 1}(1-q_\tau^nq_z^{-1}).
$$
\end{enumerate}
\end{theorem}
Note that $\wt\gamma$ differs from Kato's $\gamma$.
\begin{corollary}\label{theta-evaluation}
Let $t=\frac{a\tau}{N}+\frac{b}{N}\in \bbC/(\bbZ\tau+\bbZ)$ be an 
$N$-torsion point, $a,b\in \bbZ$ and let $\zeta_N:=e^{\frac{2\pi i}{N}}$, then 
$$
_c\vartheta(\tau,t)=
q_\tau^{\frac{1}{2}(c^2B_2(\{\frac{a}{N}\})-B_2(\{\frac{ca}{N}\}))}
(-\zeta_N^b)^{\frac{c-c^2}{2}}\frac{(1-q_\tau^{\frac{a}{N}} \zeta_N^b)^{c^2}}
{(1-q_\tau^{\frac{ca}{N}} \zeta_N^{cb})}
\frac{\wt\gamma_{q_\tau}(q_\tau^{\frac{a}{N}}\zeta_N^b)^{c^2}}
{\wt\gamma_{q_\tau}(q_\tau^{\frac{ca}{N}}\zeta_N^{cb})}
$$
\end{corollary}
\begin{proof}
This follows from Theorem \ref{modular-units} by writing
$q_z=q_\tau^{\frac{a}{N}}\zeta_N^b$ and a straightforward computation
using $B_2(x)=x^2-x+\frac{1}{6}$, so that
$$
c^2B_2(\{\frac{a}{N}\})-B_2(\{\frac{ca}{N}\})=\frac{(c-c^2)a}{N}+\frac{c^2-1}{6}.
$$
\end{proof}
For the elliptic curve
$\pi:\cE\to S$ and an integer $N>1$ consider the exact sequence
of finite \'etale group schemes
$$
0\to \cE[\ell^r]\to \cE[\ell^r N]\to \cE[N]\to 0.
$$
In the notation of \eqref{torsor-r-sequence} we have
$H_r=\cE[\ell^r]$, $G_r=\cE[\ell^rN]$ and $T=\cE[N]$.

For each section 
$t\in \cE[N](S)$ one has the $\cE[\ell^r]$-torsor
$H_r\langle t\rangle=\cE[\ell^r]\langle t\rangle$ defined
by the cartesian diagram
$$
\xymatrix{
\cE[\ell^r]\langle t\rangle\ar[r]\ar[d]_{p_{r,t}}&
\cE[\ell^rN]\ar[r]\ar[d]&
\cE\ar[d]^{[\ell^r]}\\
S\ar[r]^t&\cE[N]\ar[r]&\cE.
}
$$
We denote by $\sH_r$ and $\sH_r\langle t\rangle$ the sheaves associated
to $\cE[\ell^r]$ and $\cE[\ell^r]\langle t\rangle$ respectively. We also
define
\begin{align*}
\sH&:=(\sH_r)_{r\ge 1}\\
\sH\langle t\rangle&:=(\sH_r\langle t\rangle)_{r\ge 1}.
\end{align*}
Let $c>1$ be an integer
with $(c,6\ell N)=1$ and consider the function
$$
{_c\vartheta_\cE}:\cE\setminus \cE[c]\to \Gm.
$$
By Theorem \ref{modular-units} this function is norm-compatible
$$
[\ell^r]_*({_c\vartheta_\cE})={_c\vartheta_\cE}.
$$
Note that for $t\neq e$ one has 
$\cE[\ell^r]\langle t\rangle\subset \cE\setminus \cE[c]$, by our condition on $c$.
Thus, we can restrict $ {_c\vartheta_\cE}$ to an invertible function,
called $\theta_r$ in Definition \ref{norm-comp-def}, on 
$\cE[\ell^r]\langle t\rangle$.
\begin{definition}\label{ES-def}
Let 
$$
\cE\cS_{c,r}^{\langle t\rangle}:=\partial_r({_c\vartheta_\cE})\in
H^1(S,\Lambda_r(\sH_r\langle t\rangle)(1))
$$
and in the limit
$$
\cE\cS_c^{\langle t\rangle}:=\prolim_r\cE\cS_r^{\langle t\rangle}
\in 
H^1(S,\Lambda(\sH\langle t\rangle)(1)).
$$
\end{definition}
The section $\tau_{r,t}$ from 
Definition \ref{tau-r-t-def} is 
given by
\begin{equation*}
\tau_{r,t}:\cE[\ell^r]\langle t\rangle\hookrightarrow \cE[\ell^rN]\xrightarrow{[N]}
\cE[\ell^r].
\end{equation*}
Its $k$-tensor power gives
$$
\tau^{[k]}_{r,t}\in H^0(\cE[\ell^r]\langle t\rangle,\TSym^k\sH_r).
$$
Soul\'e's twisting construction allows now to define:
\begin{definition}\label{e-k-def} With the above notations let
\begin{equation*}
_c e_{k,r}(t):=
p_{r,t*}(\partial_r({_c\vartheta_\cE})\cup \tau_{r,t}^{[k]})\in
H^1(S,\TSym^k\sH_r(1))
\end{equation*}
and
\begin{equation*}
_c e_{k}(t):= \prolim_r {_c e}_{k,r}(t)\in H^1(S,\TSym^k\sH(1)).
\end{equation*}
We call $_c e_{k}(t)$ the \emph{elliptic Soul\'e element}.
For a function $\psi:(\cE[N](S)\setminus \{e\})\to \bbZ_\ell $ we let
$$
_c e_{k}(\psi):=\sum_{t\in \cE[N](S)\setminus \{e\}}\psi(t)_c e_{k}(t).
$$
\end{definition}

Suppose that $S$ is a scheme such that
the group scheme $\cE[\ell^rN]$ is isomorphic
to $(\bbZ/\ell^rN\bbZ)^2$ (for example $S=Y(\ell^r N)$).
Then one has 
$\sH_{r}\isom(\bbZ/\ell^r\bbZ)^2$ and the
pull-back of $_c e_{k,r}(t)$ to $S$
is given explicitly by
$$
_c e_{k,r}(t)=\sum_{[\ell^r]Q=t}\partial_r({_c\vartheta_\cE}(Q))\cup ([N]Q)^{\otimes k}\in H^1(S,\TSym^k(\bbZ/\ell^r\bbZ)^2(1)).
$$
The moment map is in our context
$$
\mom_{r,t}^k:H^1(S,\Lambda_r(\sH_r\langle t\rangle)(1))\to
H^1(S,\TSym^k\sH_r(1))
$$
or in the limit
$$
\mom_t^k:H^1(S,\Lambda(\sH\langle t\rangle)(1))\to
H^1(S,\TSym^k\sH(1)).
$$
From the general result Proposition \ref{mom-and-twist} we get:
\begin{proposition}\label{ell-soule-as-mom}
One has
$$
\mom_{r,t}^k(\cE\cS^{\langle t\rangle}_{c,r})={_c e}_{k,r}(t)
$$
and
$$
\mom_t^k(\cE\cS^{\langle t\rangle}_{c})={_c e}_{k}(t).
$$
\end{proposition}
For later use we note:
\begin{lemma}\label{e-of-minus-t}
With the above notations, one has the relation
$$
{_c e}_{k}(-t)=(-1)^k{_c e}_{k}(t).
$$
\end{lemma}
\begin{proof}
The norm-compatibility of ${_c\vartheta_\cE}$ implies
$[-1]_*{_c\vartheta_\cE}={_c\vartheta_\cE}$ and hence
$$
[-1]_*\cE\cS^{\langle -t\rangle}_{c,r}=\cE\cS^{\langle t\rangle}_{c,r}.
$$
The claim follows from the commutative diagram
$$
\begin{CD}
\Lambda_r(\sH_r\langle -t\rangle)@>{[-1]_!}>> \Lambda_r(\sH_r\langle t\rangle)\\
@V\mom_{-t}^kVV@VV\mom_t^kV\\
\TSym^k\sH_r@>(-1)^k>>\TSym^k\sH_r.
\end{CD}
$$
\end{proof}

%
%
\section{Eisenstein classes, elliptic Soul\'e elements and
the integral $\ell$-adic elliptic polylogarithm}
%
%
In this section we compare the elliptic Soul\'e elements
with the Eisenstein classes. The idea consists in writing the
Eisenstein classes as specializations of the elliptic polylogarithm
and then to define an integral version of the elliptic polylogarithm
which is directly related to the elliptic Soul\'e elements.

%
\subsection{A brief review of the elliptic logarithm sheaf}\label{log-section}
%

We give a brief review of the elliptic polylogarithm and refer for more
details to \cite{Huber-pune}, the original source \cite{BeLe} or to
the appendix A in \cite{Huber-Kings99}.

Let $\pi:\cE\to S$ be a family of elliptic curves with unit 
section $e:S\to \cE$. We let 
\begin{equation}
\Lambda:=\bbZ/\ell^r\bbZ, \bbZ_\ell, \bbQ_\ell
\end{equation}
and we consider lisse sheaves of $\Lambda$-modules.

A $\Lambda$-sheaf 
$\sG$ is unipotent of length $n$ with respect to $\pi$, if it
has a filtration $\sG=A^0\sG\supset A^\sG\supset \ldots\supset A^{n+1}\sG=0$
such that $A^k\sG/A^{k+1}\sG\isom \pi^*\sF^k$ for a lisse
$\Lambda$-sheaf $\sF^k$ on $S$. 
Beilinson and Levin show:
\begin{proposition}[\cite{BeLe} Proposition 1.2.6] There is
a $n$-unipotent sheaf $\Log_{\Lambda}^{(n)}$ together with a section
$1^{(n)}\in \Gamma(S,e^*\Log_{\Lambda}^{(n)})$ such that for any $n$-unipotent
$\Lambda$-sheaf $\sG$ the homomorphism
\begin{align*}
\pi_*\ul\Hom_\cE(\Log_{\Lambda}^{(n)},\sG)&\to e^*\sG\\
\phi& \mapsto \phi\verk 1^{(n)}
\end{align*}
is an isomorphism. The pair $(\Log_{\Lambda}^{(n)},1^{(n)})$ is unique up to 
unique isomorphism.
\end{proposition}
Obviously, as any $n-1$-unipotent sheaf is also $n$-unipotent, one
has transition maps $\Log_{\Lambda}^{(n)}\to \Log_{\Lambda}^{(n-1)}$ which map 
$1^{(n)}\mapsto 1^{(n-1)}$. 
\begin{definition}
The pro-sheaf $(\Log_{\Lambda},1):=(\Log_{\Lambda}^{(n)},1^{(n)})$ with the above transition 
maps is called the \emph{elliptic logarithm}
sheaf.
\end{definition}
We review some facts about $\Log_\Lambda$. Let 
$\sH_{\Lambda}:=\ul\Hom_S(R^1\pi_*\Lambda,\Lambda)$, 
then one has exact sequences
\begin{equation}
0\to \pi^*\Sym^n\sH_{\Lambda}\to \Log_{\Lambda}^{(n)}\to \Log_{\Lambda}^{(n-1)}\to 0,
\end{equation}
which in the case that $\Lambda=\bbQ_\ell$ induce an isomorphism
\begin{equation}
e^*\Log_{\bbQ_\ell}^{(n)}\isom \prod_{k=0}^n\Sym^k\sH_{\bbQ_\ell},
\end{equation}
which maps $1^{(n)}$ to $1\in \bbQ_\ell$. Also in the case 
$\Lambda=\bbQ_\ell$ the sheaf $\Log_{\bbQ_\ell}$ admits an action 
of $\sH_{\bbQ_\ell}$
\begin{equation}\label{mult-def}
\mult:
\pi^*\sH_{\bbQ_\ell}\otimes \Log_{\bbQ_\ell}\to \Log_{\bbQ_\ell},
\end{equation}
which on the associated graded pieces $\pi^*\Sym^k\sH_{\bbQ_\ell}$ 
is just the usual multiplication with $\sH_{\bbQ_\ell}$
$$
\pi^*\sH_{\bbQ_\ell}\otimes\pi^*\Sym^k\sH_{\bbQ_\ell}\to
\pi^*\Sym^{k+1}\sH_{\bbQ_\ell}.
$$
The most important fact about the logarithm sheaf is the vanishing 
of its higher direct images except the second one. 
\begin{proposition}[\cite{BeLe}, Lemma 1.2.7]\label{log-coh}
One has
$$
R^i\pi_*\Log_{\Lambda}=\begin{cases}
0&\mbox{  if }i\neq 2\\
R^2\pi_*\Lambda\isom \Lambda(-1) &\mbox{  if }i= 2.
\end{cases}
$$
\end{proposition}

Another important fact about the logarithm sheaf is the splitting 
principle, which we formulate as follows:
\begin{proposition}[\cite{BeLe} 1.2.10 (vi), \cite{Huber-Kings99} Corollary A.2.6.]\label{log-splitting}
Let $\varphi:\cE\to \cE'$ be an isogeny and denote by 
$\Log'_{\bbQ_\ell}$ the logarithm sheaf of $\cE'$. Then 
one has an isomorphism 
$$
\Log_{\bbQ_\ell}\isom \varphi^*\Log'_{\bbQ_\ell}.
$$
In particular, for each section $t\in \ker\varphi(S)$ one has a canonical
isomorphism 
$$
t^*\Log_{\bbQ_\ell}\isom {e'}^*\Log'_{\bbQ_\ell}\isom e^*\Log_{\bbQ_\ell}=
\prod_{n\ge 0}\Sym^n\sH_{\bbQ_\ell}.
$$
\end{proposition}
Note that in the case where $\varphi=[N]$ the isomorphism in the proposition
induces the multiplication by $[N]^k$ on the graded pieces
$\pi^*\Sym^k\sH_{\bbQ_\ell}$ of $\Log_{\bbQ_\ell}$.

%
\subsection{The elliptic polylogarithm and Eisenstein classes}
%
The Leray spectral sequence together with Proposition \ref{log-coh} and 
the localization sequence give
an isomorphism 
\begin{equation}\label{ext-for-pol-def}
\Ext^1_{\cE\setminus \{e\}}(\pi^*\sH_{\bbQ_\ell},\Log_{\bbQ_\ell}(1))\isom
\Hom_S(\sH_{\bbQ_\ell},\prod_{n\ge 1}\Sym^n\sH_{\bbQ_\ell})
\end{equation}
(see \cite[A.3]{Huber-Kings99} or \cite{Huber-pune}).
\begin{definition}\label{pol-def}
The \emph{(small) elliptic polylogarithm} is the class
\begin{equation*}
\pol\in \Ext^1_{\cE\setminus \{e\}}(\pi^*\sH_{\bbQ_\ell},\Log_{\bbQ_\ell}(1)),
\end{equation*}
which maps to the canonical inclusion 
$\sH_{\bbQ_\ell}\hookrightarrow\prod_{n\ge 1}\Sym^n\sH_{\bbQ_\ell}$ under the
above isomorphism \eqref{ext-for-pol-def}. 
\end{definition}

Consider a non-zero $N$-torsion section $t\in \cE(S)$. 
If we use the isomorphism
$t^*\Log_{\bbQ_\ell}\isom \prod_{n\ge 0}\Sym^n\sH_{\bbQ_\ell}$ from
Proposition \ref{log-splitting} we get
$$
t^*\pol=(t^*\pol^n)_{n\ge 0}\in \Ext^1_S(\sH_{\bbQ_\ell},\prod_{n\ge 0}\Sym^n\sH_{\bbQ_\ell}(1)).
$$
To get classes in $H^1(S,\Sym^n\sH_{\bbQ_\ell}(1))$ we use the map
\begin{equation}\label{contr-def}
\contr_{\sH_{\bbQ_\ell}}:
\Ext^1_S(\sH_{\bbQ_\ell},\prod_{n\ge 0}\Sym^n\sH_{\bbQ_\ell}(1))\to
\Ext^1_S({\bbQ_\ell},\prod_{n\ge 1}\Sym^{n-1}\sH_{\bbQ_\ell}(1))
\end{equation}
defined by first tensoring an extension with $\sH_{\bbQ_\ell}^\vee$,
where $\sH_{\bbQ_\ell}^\vee$ is the dual of $\sH_{\bbQ_\ell}$, and then 
compose with the contraction map
$$
\sH_{\bbQ_\ell}^\vee\otimes\Sym^n\sH_{\bbQ_\ell}
\to \Sym^{n-1}\sH_{\bbQ_\ell}
$$
mapping $h^\vee\otimes h_1\otimes\cdots \otimes h_n$ to 
$\frac{1}{n+1}\sum_{j=1}^n h^\vee(h_j)h_1\otimes\cdots \wh{h_j}\cdots\otimes h_n$.
\begin{definition}\label{Eis-def}
Let $N>1$ and $t\in \cE[N](S)$ be a non-zero $N$-torsion point,
then
$$
\Eis_{\bbQ_\ell}^k(t):=-N^{k-1}\contr_{\sH_{\bbQ_\ell}}(t^*\pol^{k+1})
\in H^1(S,\Sym^k\sH_{\bbQ_\ell}(1))
$$
is called the \emph{$k$-th Eisenstein class}. If $\psi:(\cE[N](S)\setminus \{e\})\to \bbQ$
is a map, we define
$$
\Eis_{\bbQ_\ell}^k(\psi):=
\sum_{t\in \cE[N](S)\setminus \{e\}}\psi(t)\Eis_{\bbQ_\ell}^k(t).
$$
\end{definition}
\begin{remark}
The factor $-N^{k-1}$ is for historical reasons as the Eisenstein classes were
originally defined in a different way by Beilinson 
(see \cite[Theorem 7.3]{Beilinson-modular-curves})
\end{remark}

%
\subsection{A variant of the elliptic polylogarithm}
%

For the comparison of the Eisenstein classes with the elliptic 
Soul\'e elements
a slight variant of the elliptic polylogarithm is useful. 

The localization sequence for $\Log_{\bbQ_\ell}$ on $\cE$ and
the closed subscheme $\cE[c]$ for $c> 1$ gives
\begin{equation}\label{c-loc-seq}
0\to H^1(\cE\setminus \cE[c],\Log_{\bbQ_\ell}(1))\xrightarrow{\res}
H^0(\cE[c], \Log_{\bbQ_\ell}\mid_{\cE[c]})\to 
\bbQ_\ell\to 0
\end{equation}
because $H^1(\cE,\Log_{\bbQ_\ell}(1))=0$ and 
$H^2({\cE},\Log_{\bbQ_\ell}(1))\isom H^2(\cE,\bbQ_\ell)\isom\bbQ_\ell$ by 
Proposition  \ref{log-coh}. In $H^0(\cE[c], \Log_{\bbQ_\ell}\mid_{\cE[c]})$ 
we have an element which maps to $0$ in $\bbQ_\ell$ as follows:
We have
$$
H^0(\cE[c],\bbQ_\ell)\subset  H^0(\cE[c], \Log_{\bbQ_\ell}\mid_{\cE[c]})
$$
and consider $q:\cE[c]\to S$ and the section $e:S\to \cE[c]$. These morphisms
induce 
$$
e_*:H^0(S,\bbQ_\ell)\to H^0(\cE[c],\bbQ_\ell)
$$
and 
$$
q^*:H^0(S,\bbQ_\ell)\to H^0(\cE[c],\bbQ_\ell).
$$
\begin{definition}
Let $1\in H^0(S,\bbQ_\ell)$ be the constant section which is identically
$1$ on $S$. Then we let 
$$
c^2e_*(1)-q^*(1)\in H^0(\cE[c],\bbQ_\ell)\subset 
H^0(\cE[c], \Log_{\bbQ_\ell}\mid_{\cE[c]}).
$$
\end{definition}
Note that $c^2e_*(1)-q^*(1)$ maps to $0\in \bbQ_\ell$ under the
map in \eqref{c-loc-seq}. We can now define the variant of the
elliptic polylogarithm.

\begin{definition}\label{varpol-def}
The elliptic polylogarithm $\pol_c$ associated to $c^2e_*(1)-q^*(1)$
is the cohomology class
\begin{equation*}
\pol_{c}\in 
H^1(\cE\setminus \cE[c],\Log_{\bbQ_\ell}(1))
\end{equation*}
with $\res(\pol_c)=c^2e_*(1)-q^*(1)\in H^0(\cE[c], \Log_{\bbQ_\ell}\mid_{\cE[c]})$. 
\end{definition}
This cohomology class is related to $\pol$ as follows. Write
$$
H^1(\cE\setminus \cE[c],\Log_{\bbQ_\ell}(1))\isom 
\Ext^1_{\cE\setminus \cE[c]}(\bbQ_\ell,\Log_{\bbQ_\ell}(1))
$$
and define a map 
\begin{equation}\label{mult-sH-def}
\mult_{\sH_{\bbQ_\ell}}:
\Ext^1_{\cE\setminus \cE[c]}(\bbQ_\ell,\Log_{\bbQ_\ell}(1))\to
\Ext^1_{\cE\setminus \cE[c]}(\pi^*\sH_{\bbQ_\ell},\Log_{\bbQ_\ell}(1))
\end{equation}
by first tensoring an extension with $\pi^*\sH_{\bbQ_\ell}$
and then push-out with 
$\mult:\pi^*\sH_{\bbQ_\ell}\otimes\Log_{\bbQ_\ell}\to \Log_{\bbQ_\ell}$ 
from Equation \eqref{mult-def}. This gives 
$$
\mult_{\sH_{\bbQ_\ell}}( \pol_{c})\in
\Ext^1_{\cE\setminus \cE[c]}(\pi^*\sH_{\bbQ_\ell},\Log_{\bbQ_\ell}(1)).
$$
On the other hand consider 
$$
[c]^*\pol\in \Ext^1_{\cE\setminus \cE[c]}
(\pi^*\sH_{\bbQ_\ell},[c]^*\Log_{\bbQ_\ell}(1))
$$
and use the isomorphism $\Log_{\bbQ_\ell}\isom [c]^*\Log_{\bbQ_\ell}$ 
from Proposition \ref{log-splitting} to obtain a class in 
$\Ext^1_{\cE\setminus \cE[c]}(\pi^*\sH_{\bbQ_\ell},\Log_{\bbQ_\ell}(1))$.
Restriction of  $\pol$ to $\cE\setminus \cE[c]$ gives another class in
$\Ext^1_{\cE\setminus \cE[c]}(\pi^*\sH_{\bbQ_\ell},\Log_{\bbQ_\ell}(1))$.
\begin{proposition}\label{pol-comparison}
There is an equality 
$$
\mult_{\sH_{\bbQ_\ell}}(\pol_{c})=
c^2\pol\mid_{\cE\setminus \cE[c]}-c[c]^*\pol
$$
in 
$\Ext^1_{\cE\setminus \cE[c]}(\pi^*\sH_{\bbQ_\ell},\Log_{\bbQ_\ell}(1))$.
\end{proposition}
\begin{proof}
As in Equation \eqref{ext-for-pol-def} we have 
$$
\Ext^1_{\cE\setminus \cE[c]}(\pi^*\sH_{\bbQ_\ell},\Log_{\bbQ_\ell}(1))\subset
\Hom_{\cE[c]}(\sH_{\bbQ_\ell},\prod_{n\ge 0}\Sym^n\sH_{\bbQ_\ell})
$$
and we have to show that the images of the elements
$\mult_{\sH_{\bbQ_\ell}}(\pol_{c})$ and
$c^2\pol\mid_{\cE\setminus \cE[c]}-c[c]^*\pol$ in the right hand side
are the same. One has two maps 
$$
e_*:\Hom_S(\sH_{\bbQ_\ell},\sH_{\bbQ_\ell})\to
\Hom_{\cE[c]}(\sH_{\bbQ_\ell},\sH_{\bbQ_\ell})
$$
and 
$$
q^*:\Hom_S(\sH_{\bbQ_\ell},\sH_{\bbQ_\ell})\to
\Hom_{\cE[c]}(\sH_{\bbQ_\ell},\sH_{\bbQ_\ell}).
$$
It follows from the definition
of $\mult_{\sH_{\bbQ_\ell}}(\pol_{c})$ and
$c^2\pol\mid_{\cE\setminus \cE[c]}-c[c]^*\pol$ that both elements map
to 
$$
c^2e_*(\id)-q^*(\id)\in 
\Hom_{\cE[c]}(\sH_{\bbQ_\ell},\sH_{\bbQ_\ell})
$$ 
(note that the identification $\Log_{\bbQ_\ell}\isom [c]^*\Log_{\bbQ_\ell}$
is multiplication by $c$ on $\sH_{\bbQ_\ell}$ so that the residue of
$[c]^*\pol$ is $\frac{1}{c}\id_{\cE[c]}$).
\end{proof}
\subsection{The variant of the elliptic polylogarithm and Eisenstein classes}
We are going to explain how specializations of $\pol_c$ are related to
the Eisenstein classes.

Let $(c,N)=1$ and recall from Definition \ref{varpol-def} the class
$$
\pol_{c}\in 
\Ext^1_{\cE\setminus \cE[c]}(\bbQ_\ell,\Log_{\bbQ_\ell}(1)).
$$
If we pull this back along a non-zero $N$-torsion section $t\in \cE[N](S)$,
we get, using again $t^*\Log_{\bbQ_\ell}\isom\prod_{n\ge 0}\Sym^n\sH_{\bbQ_\ell}$,
$$
t^*\pol_{c}\in 
\Ext^1_S(\bbQ_\ell, \prod_{n\ge 0}\Sym^n\sH_{\bbQ_\ell}(1))
$$
and the $k$-th component gives a class
$$
t^*\pol_{c}^k\in H^1(S,\Sym^k\sH_{\bbQ_\ell}(1)).
$$
\begin{proposition}\label{varpol-Eis-comparison}
In $H^1(S,\Sym^k\sH_{\bbQ_\ell}(1))$ we have the equality
$$
t^*\pol_{c}^k=
\frac{-1}{N^{k-1}}(c^2\Eis_{\bbQ_\ell}^k(t)-c^{-k}\Eis_{\bbQ_\ell}^k({[c]t})).
$$
In particular, for $c\equiv 1\mod{N}$ one has
$$
t^*\pol_{c}^k=
\frac{-1}{N^{k-1}}\frac{c^{k+2}-1}{c^k}\Eis_{\bbQ_\ell}^k(t).
$$
\end{proposition}
\begin{proof}
According to Proposition \ref{pol-comparison} we have
$$
\mult_{\sH_{\bbQ_\ell}}(\pol_{c}^k)=
c^2\pol^{k+1}\mid_{\cE\setminus \cE[c]}-c[c]^*\pol^{k+1}.
$$
Taking the pull-back along $t$ of the right hand side and
applying the map $\contr_{\sH_{\bbQ_\ell}}$ gives 
$$
\frac{-1}{N^{k-1}}(c^2\Eis_{\bbQ_\ell}^k(t)-c^{-k}\Eis_{\bbQ_\ell}^k({[c]t}))
$$
(note that the isomorphism $\Log_{\bbQ_\ell}\isom [c]^*\Log_{\bbQ_\ell}$
is multiplication by $c^{k+1}$ on $\Sym^{k+1}\sH_{\bbQ_\ell}$ by
the remark after Proposition \ref{log-splitting} so that we have to divide
by $c^{k+1}$).
Thus it remains to show that
$$
\contr_{\sH_{\bbQ_\ell}}
(t^*\mult_{\sH_{\bbQ_\ell}}( \pol_{c}^k))=
t^*\pol_{c}^k.
$$
But obviously we have 
$\contr_{\sH_{\bbQ_\ell}}\verk t^*\mult_{\sH_{\bbQ_\ell}}=
\contr_{\sH_{\bbQ_\ell}}\verk\mult_{\sH_{\bbQ_\ell}}t^*$, where
the last $\mult_{\sH_{\bbQ_\ell}}$ is now on $\prod_{n\ge 0}\Sym^n\sH_{\bbQ_\ell}$,
which gives 
$$
\contr_{\sH_{\bbQ_\ell}} (t^*\mult_{\sH_{\bbQ_\ell}}( \pol_{c}^k))=
\contr_{\sH_{\bbQ_\ell}}\verk\mult_{\sH_{\bbQ_\ell}}( t^*\pol_{c}^k).
$$
A direct computation shows that 
$\contr_{\sH_{\bbQ_\ell}}\verk\mult_{\sH_{\bbQ_\ell}}$ is the identity map. 
This gives the desired result. 
\end{proof}

\subsection{Sheaves of Iwasawa modules and the elliptic 
logarithm sheaf}
In this section we relate the elliptic logarithm sheaf to 
a certain sheaf of Iwasawa modules. 

Write $\cE_r:=\cE$ with structure map $\pi_{r*}:\cE_r\to S$
and identity section $e_r:S\to \cE_r$. Let 
$$
p_r:=[\ell^r]:\cE_r\to \cE
$$
be the $\ell^r$-multiplication map.
\begin{definition}
Let $\Lambda_n=\bbZ/\ell^n\bbZ$, 
then the \emph{geometric elliptic logarithm sheaf}
with coefficients in $R$
is the inverse system 
$$
\sL_{\Lambda_n}:=(p_{r*}\Lambda_n)
$$
where the transition maps are the trace maps 
$p_{r+1*}\sL_n\to p_{r*}\sL_n$.
Define a ring sheaf by 
$$
\sR_{\Lambda_n}:=e^*\sL_{\Lambda_n}
$$
and let $1_n\in \sR_{\Lambda_n}$ be the identity section.
\end{definition} 
As $\cE_r$ is an $\cE[\ell^r]$-torsor over $\cE$, the sheaf 
$\sL_{\Lambda_n}$ is an $\pi^*\sR_{\Lambda_n}$-module, which
is locally free of rank one. Denote by $\sI_{\Lambda_n}\subset\sR_{\Lambda_n}$
the augmentation ideal and by $\sI_{\Lambda_n}^k$ its $k$-th power. We define
\begin{equation}
\sL^{(k)}_{\Lambda_n}:=\sL_{\Lambda_n}\otimes_{\pi^*\sR_{\Lambda_n}} \pi^*(\sR_{\Lambda_n}/\sI_{\Lambda_n}^{k+1}).
\end{equation}
The first main result in this section is the following theorem:
\begin{theorem}\label{log-lambda-n-comp}
There is a canonical isomorphism 
$$
\sL_{\Lambda_n}^{(k)}\xrightarrow{\isom}\Log_{\Lambda_n}^{(k)}
$$
which maps $1_n$ to $1^{(k)}_{\Lambda_n}$. Here $\Log_{\Lambda_n}^{(k)}$ denotes
the constant inverse system. 
\end{theorem}
For the proof we need a characterization of lisse $\Lambda_n$-sheaves on $\cE$.
\begin{proposition}\label{unip-char}
Let $\sG$ be a lisse $\Lambda_n$-sheaf on $\cE$. Then there is an integer
$s$ such that
$$
\pi_*\ul\Hom_\cE(p_{s*}{\Lambda_n},\sG)\isom e^*\sG.
$$
In particular, the functor $\sG\mapsto e^*\sG$ induces 
an equivalence between the category of lisse $\Lambda_n$-sheaves on $\cE$
and lisse $\Lambda_n$-sheaves on $S$ with a continuous action of
$\sR_{\Lambda_n}$.
\end{proposition}
\begin{proof}
As $p_r$ is finite \'etale, one has
$$
\pi_*\ul\Hom_{\cE}(p_{r*}\Lambda_n,\sG)=
\pi_{r*}\ul\Hom_{\cE_r}(\Lambda_n,p_r^*\sG)=\pi_{r*}p_r^*\sG.
$$
As $\sG$ is a lisse $\Lambda_n$-sheaf, there is an $s$ such that
$p_s^*\sG$ comes from $S$, which means $p_s^*\sG\isom \pi_s^*e_s^*p_s^*\sG$.
This implies that 
$$
\pi_*\ul\Hom_{\cE}(p_{s*}\Lambda_n,\sG)=\pi_{s*}\pi_s^*e_s^*p_s^*\sG \isom
e^*\sG.
$$
This isomorphism allows to define a continuous action of $\sR_{\Lambda_n}$
on $e^*\sG$ (where continuous means that 
the action factors through some $e^*p_{s*}\Lambda_n$). The inverse functor is
given by $\sF\mapsto \pi^*\sF\otimes_{\pi^*\sR_{\Lambda_n}}\sL_{\Lambda_n}$.
\end{proof}
\begin{proof}[Proof of Theorem \ref{log-lambda-n-comp}]
From Proposition \ref{unip-char} we get a morphism of pro-sheaves
$\sL_{\Lambda_n}\to \Log_{\Lambda_n}^{(k)}$
corresponding to $1^{(k)}_{\Lambda_n}$. It also follows that 
$\sL_{\Lambda_n}^{(k)}$ is $k$-unipotent because the $\sR_{\Lambda_n}$-module
structure on $e^*\sL_{\Lambda_n}^{(k)}$ factors through
$\sR_{\Lambda_n}/\sI_{\Lambda_n}^{k+1}$. In particular, the
above morphism factors through $\sL_{\Lambda_n}^{(k)}$. Moreover,
by the definition of 
$\Log_{\Lambda_n}^{(k)}$ we get also a morphism 
$\Log_{\Lambda_n}^{(k)}\to \sL_{\Lambda_n}^{(k)}$
corresponding to $1_n$. It is straightforward to check that
these two morphisms are inverse to each other. 
\end{proof}
\begin{definition}
Define the pro-sheaf $\sL$  by
$$
\sL:=(p_{r*}\Lambda_r)_{r\ge 1}
$$
where the transition maps are induced by the trace maps and the
reduction modulo $\ell^r$. We also let $\sR:=e^*\sL$ with unit section $1$ and
denote by $\sI\subset\sR$ the augmentation ideal. Finally, let
$$
\sL^{(k)}:=\sL\otimes_{\pi^*\sR}\pi^*(\sR/\sI^{k+1}).
$$
\end{definition}
\begin{theorem}\label{log-comp}
There is a canonical isomorphism 
$$
\sL^{(k)}\xrightarrow{\isom}\Log_{\bbZ_\ell}^{(k)}
$$
which maps $1$ to $1^{(k)}$, where $\Log_{\bbZ_\ell}^{(k)}$
is  $(\Log_{\Lambda_r}^{(k)})_{r\ge 1}$.
\end{theorem}
\begin{proof}
There is a surjective morphism
$\sL\to \sL_{\Lambda_n}$ and we get with Theorem \ref{log-lambda-n-comp}
a map
$$
\sL\to \sL_{\Lambda_n}\to \sL_{\Lambda_n}^{(k)}\isom \Log_{\Lambda_n}^{(k)},
$$
which induces a morphism $\sL\to \Log_{\bbZ_\ell}^{(k)}$. As already the
map $\sL\to \sL_{\Lambda_n}^{(k)}$ factors through $\sL^{(k)}$, we
get the desired morphism $\sL^{(k)}\to\Log_{\bbZ_\ell}^{(k)}$, which
is surjective by construction. From the isomorphism 
$\sL^{(k)}\otimes_{\bbZ_\ell}\Lambda_n\isom \sL_{\Lambda_n}^{(k)}$
one deduces that $\sL^{(k)}$ is $k$-unipotent. By the definition of
$\Log_{\bbZ_\ell}^{(k)}$ we get a morphism in the other direction
$\Log_{\bbZ_\ell}^{(k)}\to \sL^{(k)}$ and one checks directly that
this is inverse to $\sL^{(k)}\to\Log_{\bbZ_\ell}^{(k)}$.
\end{proof}
We now discuss the relation between the sheaves $\sL$ and the
sheaves of Iwasawa modules. 
\begin{proposition}\label{log-and-iwasawa}
Let $t:S\to \cE$ be an $N$-torsion section, then
$$
t^*\sL\isom \Lambda(\sH\langle t\rangle).
$$
\end{proposition}
\begin{proof}
From the commutative diagram
$$
\begin{CD}
\cE[\ell^r]\langle t\rangle@>>> \cE_r\\
@Vp_{r,t}VV@VV[\ell^r]=p_r V\\
S@>t>>\cE
\end{CD}
$$
we get $t^*p_{r*}\Lambda_r\isom p_{r,t*}\Lambda_r$ and the result
follows from the definitions.
\end{proof}
Finally, we relate the moment map for $\Lambda(\sH\langle t\rangle)$
to the splitting principle for $\Log_{\bbQ_\ell}$. It follows
from Proposition \ref{log-and-iwasawa}
and Theorem \ref{log-comp} that we have a morphism
\begin{equation}\label{comp-k-map}
H^1(S, \Lambda(\sH\langle t\rangle)(1))\to 
H^1(S,t^*\Log_{\bbZ_\ell}^{(k)}(1))
\to
H^1(S,t^*\Log_{\bbQ_\ell}^{(k)}(1)),
\end{equation}
where the last morphism is the canonical map.
\begin{definition}\label{comp-def}
The \emph{comparison map} is the inverse limit over $k$ 
of the maps in \eqref{comp-k-map}:
$$
\comp:
H^1(S, \Lambda(\sH\langle t\rangle)(1))\to
H^1(S,t^*\Log_{\bbQ_\ell}(1)).
$$
\end{definition}  
Recall from
Proposition \ref{log-splitting} the isomorphism
$$
t^*\Log_{\bbQ_\ell}\isom \prod_{n\ge 0}\Sym^n\sH_{\bbQ_\ell}.
$$
\begin{proposition}\label{mom-and-split} 
Let $t$ be an $N$-torsion section $t:S\to \cE$.
There is a commutative diagram
$$
\xymatrix{
H^1(S,\Lambda(\sH\langle t\rangle)(1) )\ar[r]^{\comp}\ar[ddr]_{\wt\mom^k_t}&
H^1(S, t^*\Log_{\bbQ_\ell}(1) )\ar[d]^{\isom}\\
&H^1(S, \prod_{n\ge 0}\Sym^n\sH_{\bbQ_\ell}(1))\ar[d]^{\pr_k}\\
&
H^1(S, \Sym^k\sH_{\bbQ_\ell}(1)),
}
$$
with $\wt\mom_t^k$ as in Definition \ref{wtmom-def}. 
\end{proposition}
\begin{proof}
As the diagram
$$
\begin{CD}
\Lambda(\sH\langle t\rangle)@>>> t^*\Log_{\bbZ_\ell}^{(k)}\\
@V[N]_!VV@VV[N]_*V\\
\Lambda(\sH)@>>>e^*\Log_{\bbZ_\ell}^{(k)}
\end{CD}
$$
commutes, it suffices to treat the case $t=0$. But recall
that the identification $t^*\Log_{\bbQ_\ell}\isom \prod_{n\ge 0}\Sym^n\sH_{\bbQ_\ell}$ is the composition
$$
\begin{CD}
t^*\Log_{\bbQ_\ell}@>[N]_*>\isom> e^*\Log_{\bbQ_\ell}@<[N]_*<\isom<\prod
_{n\ge 0}\Sym^n\sH_{\bbQ_\ell}.
\end{CD}
$$ 
This introduces a factor $\frac{1}{N^k}$ in front of $\Sym^k\sH_{\bbQ_\ell}$.
Let $I(\sH)\subset \Lambda(\sH)$
be the augmentation ideal. Then the isomorphism 
$$
\Lambda(\sH)/I(\sH)^{k+1}\isom e^*\Log_{\bbZ_\ell}^{(k)}
$$
induces isomorphisms of the associated graded pieces, which
are the $\Sym^n\sH$ for $n=0,\ldots,k$. Therefore 
\begin{multline*}
H^1(S,\Sym^k\sH(1))\to H^1(S, \Lambda(\sH)/I(\sH)^{k+1}(1))\isom
H^1(S,e^*\Log_{\bbZ_\ell}^{(k)})\\
\to
H^1(S,e^*\Log_{\bbQ_\ell}^{(k)})
\xrightarrow{\pr_k}
H^1(S,\Sym^k\sH_{\bbQ_\ell}(1))
\end{multline*}
is just the comparison map for $\Sym^k\sH$. It therefore follows from 
Lemma \ref{sym-mom-compatibility} that the diagram
$$
\xymatrix{
H^1(S, \Lambda(\sH)(1))\ar[r]\ar[dr]_{\wt\mom_t^k}& H^1(S,e^*\Log_{\bbQ_\ell}(1))\ar[d]^{\pr_k}\\
&H^1(S,\Sym^k\sH_{\bbQ_\ell}(1))
}
$$
commutes.
\end{proof}

%
\subsection{The elliptic polylogarithm and elliptic units}
%
In this section we describe the elliptic polylogarithm in terms
of Kato's norm compatible elliptic units. This will result
in a comparison of the Eisenstein classes with the elliptic Soul\'e elements.

Let $c$ be a positive integer
with $(c,6\ell N)=1$. We continue to write $\Lambda_r:=\bbZ/\ell^r\bbZ$
and $\cE_r:=\cE$ which we consider as an \'etale cover over 
$\cE$ via $p_r:=[\ell^r]:\cE_r\to\cE$.
This induces a morphism
$$
p_r:\cE_r\setminus  \cE[\ell^r c]\to \cE\setminus \cE[c].
$$
On $\cE\setminus \cE[\ell^r c]$ we have the elliptic unit ${ _c\vartheta_{\cE}}$
from Theorem \ref{modular-units}.
We denote by 
\begin{equation*}
\Theta_{c,r}:=\partial_r ({ _c\vartheta_{\cE}})\in
H^1(\cE\setminus \cE[\ell^rc],\Lambda_r(1))\isom 
H^1(\cE\setminus\cE[c],\sL_{\Lambda_r}(1))
\end{equation*}
the image of ${ _c\vartheta_{\cE}}$ under the Kummer map $\partial_r$.
As the functions ${ _c\vartheta_{\cE}}$ are norm-compatible, we can
pass to the inverse limit.
\begin{definition}
We denote by 
\begin{align*}
\begin{split}
\Theta_c:=\prolim_r \Theta_{c,r}\in
&\prolim_r H^1(\cE\setminus \cE[c],\sL_{\Lambda_r}(1))\\
&=H^1(\cE\setminus \cE[c], \sL(1))
\end{split}
\end{align*}
the inverse limit of the classes $\Theta_{c,r}$.
\end{definition}
Recall from Definition \ref{ES-def} the class 
$$
\cE\cS^{\langle t\rangle}_c\in
H^1(S, \Lambda(\sH\langle t\rangle)(1))
$$
and from Proposition \ref{log-and-iwasawa} the isomorphism
$t^*\sL\isom \Lambda(\sH\langle t\rangle)$.
\begin{lemma}\label{ES-Theta-comp}
Let  $t:S\to\cE$ be an $N$-torsion section. Then
$$
t^*\Theta_c=\cE\cS_c^{\langle t\rangle}\in
H^1(S,\Lambda(\sH\langle t\rangle)(1)).
$$
\end{lemma}
\begin{proof}
We have $t^*\Theta_c=\partial_r({_c\vartheta_\cE}\mid_{\cE[\ell^r]\langle t\rangle})$
so that the formula is clear from the definitions.
\end{proof}
As in  Definition \ref{comp-def} one can define a comparison homomorphism
\begin{equation}\label{Ext-comparison}
\comp:H^1(\cE\setminus \cE[c],\sL(1))\to
H^1(\cE\setminus \cE[c],\Log_{\bbQ_\ell}(1)).
\end{equation}
Recall from Definition \ref{varpol-def} the class
$$
\pol_{c}\in 
H^1(\cE\setminus \cE[c],\Log_{\bbQ_\ell}(1)).
$$
\begin{theorem}\label{integral-pol-comparison}
Let $(c,6\ell N)=1$, then 
$$
\comp(\Theta_c)= \pol_{c}\in 
H^1(\cE\setminus \cE[c],\Log_{\bbQ_\ell}(1)).
$$
\end{theorem}
\begin{proof}
Consider the commutative diagram
$$
\xymatrix{
H^1(\cE\setminus \cE[c],\sL(1))\ar[r]^{\comp}\ar[d]_{\res}&
H^1(\cE\setminus \cE[c],\Log_{\bbQ_\ell}(1))\ar@^{(->}[d]^\res\\
H^0(\cE[c],\sL\mid_{\cE[c]})\ar[r]^{\comp}&
H^0(\cE[c],\Log_{\bbQ_\ell}\mid_{\cE[c]}).
}
$$
By definition of $\pol_{c}$ its image in 
$H^0(\cE[c],\Log_{\bbQ_\ell}\mid_{\cE[c]})$ is the element
$$
c^2e_*(1)-q^*(1)\in H^0(\cE[c],\bbQ_\ell)\subset
H^0(\cE[c],\Log_{\bbQ_\ell}\mid_{\cE[c]}).
$$
To conclude the proof of Theorem \ref{integral-pol-comparison} it
suffices to compute the image of $\comp(\Theta_c)$ in
$H^0(\cE[c],\Log_{\bbQ_\ell}\mid_{\cE[c]})$.
For this we work at finite level and use the commutative diagram
$$
\xymatrix{
H^1(\cE\setminus \cE[c\ell^r],\Lambda_r(1))\ar[r]^/.7em/\res\ar[d]^\isom &
H^0(\cE[c\ell^r],\Lambda_r)\ar[d]^\isom\\
H^1(\cE\setminus \cE[c],\sL_{\Lambda_r}(1))\ar[r]^/.7em/\res&
H^0(\cE[c],\sL_{\Lambda_r}).
}
$$
The residue of ${ _c\vartheta_{\cE}}$ is 
$$
c^2e_*(1)-q^*(1)\in H^0(\cE[c],\Lambda_r)
\subset H^0(\cE[c],\sL_{\Lambda_r})
$$
and taking the inverse limit over $r$ 
shows that $\comp(\Theta_c)$ agrees with
$\pol_{c}$ in 
$H^0(\cE[c],\Log_{\bbQ_\ell}\mid_{\cE[c]})$.
\end{proof}
%
\subsection{Eisenstein classes and elliptic Soul\'e elements}
%
In this section we finally prove the comparison result between
Eisenstein classes and elliptic Soul\'e elements which is fundamental
for the whole paper.

It follows from Theorem \ref{integral-pol-comparison} that for
$t\in \cE[N](S)\setminus \{e\}$ one has
$$
t^*\comp(\Theta_c)=t^*\pol_c\in 
\prod_{n\ge 0}H^1(S,\Sym^n\sH_{\bbQ_\ell}(1)).
$$
Denote by 
$$
\pr_k:\prod_{n\ge 0}H^1(S,\Sym^n\sH_{\bbQ_\ell}(1))\to
H^1(S,\Sym^n\sH_{\bbQ_\ell}(1))
$$
the projection onto the $k$-th component.
\begin{theorem}\label{pol-eis}
Let $t\in \cE[N](S)$ be a non-zero $N$-torsion section.
Then one has
$$
\frac{1}{N^k}{_c{e}_k(t)}=
\wt\mom^k_t(\cE\cS_c^{\langle t\rangle})=
\frac{-1}{N^{k-1}}
(c^2\Eis_{\bbQ_\ell}^k(t)-c^{-k}\Eis_{\bbQ_\ell}^k({[c]t})).
$$
\end{theorem}
\begin{proof}
As $\pr_k(t^*\comp(\Theta_c))=\pr_k(t^*\pol_c)$ by Theorem
\ref{integral-pol-comparison}, this follows from Lemma \ref{ES-Theta-comp} 
together with
Proposition \ref{varpol-Eis-comparison}.
\end{proof}
\begin{lemma}\label{Eis-parity}
The Eisenstein class $\Eis_{\bbQ_\ell}^k(t)$ is of parity $(-1)^k$, i.e.,
one has 
$$
\Eis_{\bbQ_\ell}^k(-t)=(-1)^k\Eis_{\bbQ_\ell}^k(t).
$$
In particular, $\Eis_{\bbQ_\ell}^k(\psi)=0$ if $\psi$ is not of
parity $(-1)^k$, where we say that $\psi$ has parity $(-1)^k$, if 
$\psi(-t)=(-1)^k\psi(t)$.
\end{lemma}
\begin{proof}
This follows from Lemma \ref{e-of-minus-t} and Theorem \ref{pol-eis}
for $c\equiv 1\mod{N}$. 
\end{proof}

%
%
\section{The residue at $\infty$ of the elliptic Soul\'e elements}
%
%
In this section we will compute the residue at $\infty$ of the
elliptic Soul\'e elements and hence
of the Eisenstein classes.

%
\subsection{Definition of the residue at $\infty$}
%
We are going to describe several variants of the map
$\res_\infty$.

Let $\zeta_N=e^{2\pi i/N}\in \bbC$ and consider
over  $\Spec\bbQ(\zeta_N)((q^{1/N}))$ the Tate curve 
$\cE_q$ with the level structure $\alpha:(\bbZ/N\bbZ)^2\to \cE_q[N]$ given
by $(a,b)\mapsto q^a\zeta_N^b$. The corresponding map of schemes
$\Spec\bbQ(\zeta_N)((q^{1/N}))\to Y(N)$
induces $\Spec\bbQ(\zeta_N)[[q^{1/N}]]\to X(N)$ and a hence a map
$$
\infty:\Spec\bbQ(\zeta_N)\to X(N),
$$
whose image we call the cusp $\infty$. 
%

Let $\wh{X}(N)_\infty$ be the completion of $X(N)$ at $\infty$, which
can be identified 
via the above map with $\Spec\bbQ(\zeta_N)[[q^{1/N}]]$.
We denote by $\wh{Y}(N)_\infty$ the generic
fibre of $\wh{X}(N)_\infty$ so that
$\wh{Y}(N)_\infty\isom \Spec\bbQ(\zeta_N)((q^{1/N}))$. One has a commutative
diagram
\begin{equation}\label{localization-diag}
\begin{CD}
\wh{Y}(N)_\infty@>j>>\wh{X}(N)_\infty@<\infty<<\infty\\
@VVV@VVV@VV=V\\
Y(N)@>j>>X(N) @<\infty<<\infty.
\end{CD}
\end{equation}
Note that by purity one has a canonical isomorphism 
$\infty^*R^1j_*\bbZ/\ell^r\bbZ(1)\isom \bbZ/\ell^r\bbZ$ and
that $R^2j_*\bbZ/\ell^r\bbZ(1)=0$.
\begin{definition}
We define the residue map 
$$
\res_\infty:H^i(\wh{Y}(N),\bbZ/\ell^r\bbZ(1))\to 
H^{i-1}(\infty, \bbZ/\ell^r\bbZ)
$$
to be the morphism
induced by the edge morphism of the Leray spectral sequence for $Rj_*$.
\end{definition}
Consider the Tate curve $\cE_q$ over $\wh{Y}(N)_\infty$. For each
$r\ge 1$ one has an exact sequence and a commutative diagram
\begin{equation}\label{schemes-over-completion}
\xymatrix{0\ar[r]&\mu_{\ell^r N}\ar[r]&
\cE[\ell^rN]\ar[r]^{p_r}\ar[d]^{[\ell^r]}&\bbZ/\ell^rN\bbZ\ar[r]
\ar[d]^{[\ell^r]}&0\\
&&\cE[N]\ar[r]^p&\bbZ/N\bbZ
}
\end{equation}
and $p_r$ induces a finite morphism
\begin{equation}\label{p-r-def}
p_r:\cE[\ell^r]\langle t\rangle\to Z_r\langle p(t)\rangle
\end{equation}
and hence a morphism of sheaves 
$$
p_{r!}:\Lambda_r(\sH_r\langle t\rangle)\to \Lambda_r(Z_r\langle p(t)\rangle).
$$
Note that $\Lambda_r(Z_r\langle p(t)\rangle)$ is a constant sheaf over
$\wh{X}(N)_\infty$ so that we can consider the composition
\begin{multline}\label{res-composition}
H^1(\wh{Y}(N)_\infty,\Lambda_r(\sH_r\langle t\rangle)(1))\xrightarrow{p_{r*}}
H^1(\wh{Y}(N)_\infty,\Lambda_r(Z_r\langle t\rangle)(1))
\\\xrightarrow{\res_\infty}
H^0(\infty, \Lambda_r(Z_r\langle p(t)\rangle))\isom 
\Lambda_r(Z_r\langle p(t)\rangle).
\end{multline}
\begin{definition}\label{res-r-def}
We define 
$$
\res_\infty:H^1(\wh{Y}(N)_\infty,\Lambda_r(\sH_r\langle t\rangle)(1))
\to H^0(\infty, \Lambda_r(Z_r\langle p(t)\rangle))
\isom \Lambda_r(Z_r\langle p(t)\rangle)
$$
to be the composition of the maps in \eqref{res-composition}.
We also denote by
$$
\res_\infty:H^1(\wh{Y}(N)_\infty,\Lambda(\sH\langle t\rangle)(1))
\to  \Lambda(Z\langle p(t)\rangle)
$$
the inverse limit.
\end{definition}
Over $\wh{Y}(N)_\infty$ one has also the exact sequence
\begin{equation}
0\to\bbZ_\ell(1)\xrightarrow{\iota} \sH\xrightarrow{p} \bbZ_\ell\to 0.
\end{equation}
\begin{proposition}\label{monodromy}
The subsheaf $\iota(\bbZ_\ell(1))\subset  \sH$ 
are the invariants of monodromy. In particular,  
$\iota:\bbZ_\ell(1)\to\sH$ induces an isomorphism 
$\bbZ_\ell(1)\isom \infty^*j_*\sH$ and $p:\sH\to \bbZ_\ell$ induces
$$
\infty^*R^1j_*\sH(1)\isom \infty^*R^1j_*\bbZ_\ell(1)\isom\bbZ_\ell.
$$
\end{proposition}
\begin{proof}
That $\iota(\bbZ_\ell(1))\subset  \sH$ 
are the invariants of monodromy is \cite[Expos\'e IX, Proposition 2.2.5
and (2.2.5.1)]{SGA7I}. From the long exact sequence for $\infty^*Rj_*$ we
get
$$
0\to \infty^*j_* \bbZ_\ell\to \infty^*R^1j_* \bbZ_\ell(1)\to
\infty^*R^1j_*\sH\to  \infty^*R^1j_* \bbZ_\ell\to 0.
$$
As $ \infty^*j_* \bbZ_\ell\isom \bbZ_\ell$ and
$\bbZ_\ell\isom \infty^*R^1j_* \bbZ_\ell(1)$, the first map is an
isomorphism and
one gets $\infty^*R^1j_*\sH\isom\infty^*R^1j_*\bbZ_\ell\isom \bbZ_\ell(-1)$.
\end{proof}

\begin{corollary}\label{monodromy-sym}
Over $\wh{Y}(N)_\infty$ the maps 
$\Sym^k\bbZ_\ell(1)\to \Sym^k\sH$ induced by  $\iota$ and 
$\Sym^k\sH\to \Sym^k\bbZ_\ell$ induced by $p$ give rise to isomorphisms 
\begin{align*}
\bbZ_\ell(k)\isom \infty^*j_*\Sym^k\sH&&\mbox{ and }&&
\infty^*R^1j_*\Sym^k\sH(1)\isom \bbZ_\ell.
\end{align*}
\end{corollary}
\begin{proof}
This follows by induction on $k$ from Proposition \ref{monodromy} and
the exact sequence
$$
0\to \bbZ_\ell(k)\xrightarrow{\iota}\Sym^k\sH\to 
\Sym^{k-1}\sH
\to 0.
$$
\end{proof}
\begin{definition}\label{residue-Qell-def}
The \emph{residue} at $\infty$ 
is
the morphism
\begin{equation*}
H^i(\wh{Y}(N)_\infty,\Sym^k\sH_{\bbQ_\ell}(1))\to
H^{i-1}(\infty,\bbQ_\ell)
\end{equation*}
induced from the edge morphism of the Leray spectral sequence for $Rj_*$
and the isomorphism $\infty^*R^1j_*\Sym^k\sH_{\bbQ_\ell}(1)\isom \bbQ_\ell$.
\end{definition}
In the same way one defines a residue at $\infty$
\begin{equation}\label{residue-Qell-def-variant}
\res_\infty:H^i(Y(N),\Sym^k\sH_{\bbQ_\ell}(1))\to
H^{i-1}(\infty,\bbQ_\ell),
\end{equation}
which obviously factors through  the residue map on
$H^i(\wh{Y}(N)_\infty,\Sym^k\sH_{\bbQ_\ell}(1))$.

The residue maps defined in Definition \ref{res-r-def} on finite level
and in Definition \ref{residue-Qell-def} with $\bbQ_\ell$-coefficients 
are compatible in the following sense.

\begin{lemma}\label{res-comparison}
There is a commutative diagram
$$
\begin{CD}
H^1(\wh{Y}(N)_\infty,\Lambda(\sH\langle t\rangle)(1))@>{\res_\infty}>>
H^0(\infty, \Lambda(Z\langle p(t)\rangle))\\
@V{\wt\mom_t^k}VV@VV{\wt\mom_{p(t)}^k}V\\
H^1(\wh{Y}(N)_\infty, \Sym^k\sH_{\bbQ_\ell}(1)))@>{\res_\infty}>>
H^0(\infty, \bbQ_\ell)\isom \bbQ_\ell.
\end{CD}
$$
Moreover, if one uses the isomorphism 
$H^0(\infty, \Lambda(Z\langle p(t)\rangle))\isom \Lambda(Z\langle p(t)\rangle)$
the map $\wt\mom_{p(t)}^k$ is the composition
$$
\wt\mom_{p(t)}^k:\Lambda(Z\langle p(t)\rangle)\xrightarrow{\mom_{p(t)}^k}
\bbZ_\ell\xrightarrow{\frac{1}{N^kk!}}\bbQ_\ell.
$$
\end{lemma}
\begin{proof} The functoriality of the moment map gives 
\[
\begin{CD}
\Lambda(\sH\langle t\rangle)@>p_!>> \Lambda(Z\langle p(t)\rangle)\\
@V\mom^k_tVV@VV\mom^k_{p(t)}V\\
\TSym^k\sH@>\TSym^kp>>\TSym^k\bbZ_\ell\isom \bbZ_\ell
\end{CD}
\]
and the lemma follows from the definitions if one observes
that the canonical map $\Sym^k\bbZ_\ell\to \TSym^k\bbZ_\ell$
maps the generator of $\Sym^k\bbZ_\ell$ to $k!$ times the generator
of $\TSym^k\bbZ_\ell$.
\end{proof}
Finally, we treat the compatibility of the Kummer map and the residue map.
The scheme $Z_r\langle p(t)\rangle$ over $\wh{Y}(N)_\infty$ 
is the disjoint union of copies of $\Spec\bbQ(\zeta_N)((q^{1/N}))$.
An invertible function on $ Z_r\langle p(t)\rangle$ is therefore just
a collection of units in $\bbQ(\zeta_N)((q^{1/N}))$ and one can speak
of the order of the unit in the uniformizing parameter $ q^{1/N}$.
If we denote by $\bbZ[Z_r\langle p(t)\rangle]$ the abelian group of maps 
$\varphi:Z_r\langle p(t)\rangle\to \bbZ$ one gets a homomorphism
\begin{equation}
\ord_\infty: \Gm(Z_r\langle p(t)\rangle)\to \bbZ[Z_r\langle p(t)\rangle].
\end{equation}
The norm with respect to the finite morphism 
$p_r: \cE[\ell^r]\langle t\rangle\to Z_r\langle p(t)\rangle$ induces a homomorphism 
$$
p_{r*}:\Gm(\cE[\ell^r]\langle t\rangle)\to \Gm(Z_r\langle t\rangle).
$$ 
With these notations we have:
\begin{lemma}\label{compatibility-res-Kummer}
The following diagram commutes:
$$
\xymatrix{
\Gm(\cE[\ell^r]\langle t\rangle)
\ar[r]^{\ord_\infty\verk p_*}\ar[d]^{\partial_r}&
\bbZ[Z_r\langle p(t)\rangle]\ar[d]\\
H^1(\wh{Y}(N)_\infty,\Lambda_r(\sH_r\langle t\rangle)(1))\ar[r]^/0,7cm/{\res_\infty}&
\Lambda_r[Z_r\langle p(t)\rangle].
}
$$
Here the right vertical arrow reduces the coefficients modulo $\ell^r$.
\end{lemma}
\begin{proof}
Compatibility of the  Kummer map with traces and residues.
\end{proof}

%
\subsection{Computation of the residue at $\infty$ of 
the elliptic Soul\'e element}
%
Recall the residue map 
$$
\res_\infty:H^1(\wh{Y}(N)_\infty,\Lambda_r(\sH_r\langle t\rangle)(1))
\to H^0(\infty, \Lambda_r(Z_r\langle p(t)\rangle))
\isom \Lambda_r(Z_r\langle p(t)\rangle)
$$
from Definition \ref{res-r-def} and the elements 
$$
\cE\cS_{c,r}^{\langle t\rangle}\in H^1(\wh{Y}(N)_\infty,\Lambda_r(\sH_r\langle t\rangle)(1))
$$ 
defined in \ref{ES-def} and 
$$
B_{2,c,r}^{\langle p(t)\rangle}\in \Lambda_r(Z_r\langle p(t)\rangle)
$$
defined in \ref{Bernoulli-c-r-def}.

The residue of the elliptic Soul\'e elements will be deduced 
from the following fundamental result.
\begin{theorem}\label{ES-c-r-res}
With the above notation one has
$$
\res_{\infty}(\cE\cS_{c,r}^{\langle t\rangle})=B_{2,c,r}^{\langle p(t)\rangle}.
$$
In particular, taking the inverse limit one has
$$
\res_\infty(\cE\cS^{\langle t\rangle}_c)=B_{2,c}^{\langle p(t)\rangle}.
$$
\end{theorem}
\begin{proof}
Recall that $\cE\cS_{c,r}^{\langle t\rangle}=\partial_r({_c\vartheta_{\cE}})$
so that Lemma \ref{compatibility-res-Kummer} implies that
we have to compute $\ord_\infty\verk p_*({_c\vartheta_{\cE}})$. In order to do
this we perform a base change from $\wh{Y}(N)_\infty$ to 
$\wh{Y}(\ell^rN)_\infty$. We introduce the shorter notation
$$
T_r:=\wh{Y}(\ell^rN)_\infty=\Spec \bbQ(\zeta_{\ell^rN})((q^{1/\ell^rN}))
$$
for $r\ge 0$. Over $T_r$ the scheme $\cE[\ell^r]\langle t\rangle$ is
isomorphic to the constant scheme
$$
Z_r^2\langle t\rangle :=\{(x,y)\in (\bbZ/\ell^rN)^2\mid
[\ell^r](x,y)=t\}
$$ 
and the map
$p:Z_r^2\langle t\rangle\to Z_r\langle p(t)\rangle$
is simply given by the projection $\pr_1$ 
onto the first coordinate: $(x,y)\mapsto x$.
The base change from $T_0$ to $T_r$ induces a commutative diagram
\begin{equation}\label{ord-comm-diag}
\xymatrix{
\Gm(Z_r^2\langle t\rangle_{T_r})\ar[r]^{\pr_{1*}}&
\Gm(Z_r\langle \pr_1(t)\rangle_{T_r})\ar[r]^/0,2cm/{\ord_\infty}&
\bbZ[Z_r\langle \pr_1(t)\rangle]\\
\Gm(\cE[\ell^r]\langle t\rangle_{T_0})\ar[u]\ar[r]^{p_*}&
\Gm(Z_r\langle p(t)\rangle_{T_0})\ar[u]\ar[r]^/0,2cm/{\ord_\infty}&
\bbZ[Z_r\langle p(t)\rangle]\ar[u]_{\ell^r},
}
\end{equation}
where the right vertical map is the multiplication of the coefficients
with $\ell^r$. The commutativity follows from the fact that the 
morphism $T_r\to T_0$ is ramified of degree $\ell^r$ in $q^{1/N}$.
Moreover, as $\pr_1:Z_r^2\langle t\rangle\to Z_r\langle \pr_1(t)\rangle$
is unramified one has a commutative diagram
$$
\xymatrix{\Gm(Z_r^2\langle t\rangle_{T_r})
\ar[r]^/0,2cm/{\ord_\infty}\ar[d]^{\pr_{1*}}& 
\bbZ[Z_r^2\langle t\rangle]\ar[d]^{\pr_{1!}}\\
\Gm(Z_r\langle \pr_1(t)\rangle_{T_r})\ar[r]^/0,2cm/{\ord_\infty}&
\bbZ[Z_r\langle \pr_1(t)\rangle].
}
$$
For each $(x,y)\in Z_r^2\langle t\rangle$ we
now have to calculate the order of $(x,y)^*{_c\vartheta_\cE}$ at $\infty$.
For this we can work on $Y(\ell^rN)(\bbC)$. 
By Corollary \ref{theta-evaluation} the function 
$(x,y)^*{_c\vartheta_\cE}$
is explicitly given by 
$$
q_\tau^{\frac{1}{2}(c^2B_2(\{\frac{x}{\ell^rN}\})-B_2(\{\frac{cx}{\ell^rN}\}))}
(-\zeta_{\ell^rN}^y)^{\frac{c-c^2}{2}}\frac{(1-q_\tau^{\frac{x}{\ell^rN}} \zeta_{\ell^rN}^y)^{c^2}}
{(1-q_\tau^{\frac{cx}{\ell^rN}} \zeta_{\ell^rN}^{cy})}
\frac{\wt\gamma_{q_\tau}(q_\tau^{\frac{x}{\ell^rN}}\zeta_{\ell^rN}^y)^{c^2}}
{\wt\gamma_{q_\tau}(q_\tau^{\frac{cx}{\ell^rN}}\zeta_{\ell^rN}^{cy})}.
$$
As the uniformizing parameter for $Y(\ell^rN)(\bbC)$
at $\infty$ is $q^{1/\ell^rN}_\tau$,
we get 
\begin{equation*}
\ord_\infty((x,y)^*{_c\vartheta_\cE})
=\frac{\ell^rN}{2}(c^2B_2(\{\frac{x}{\ell^rN}\})-B_2(\{\frac{cx}{\ell^rN}\}))
=B_{2,c,r}^{\langle p(t)\rangle}(x).
\end{equation*}
To compute $\pr_{1!}\verk \ord_\infty((x,y)^*{_c\vartheta_\cE})$ observe 
that for a fixed $x\in Z_r\langle \pr_1(t)\rangle$ 
there are $\ell^r$ elements $y$ with $(x,y)\in Z_r^2\langle t\rangle$.
As $\ord_\infty((x,y)^*{_c\vartheta_\cE})$ is independent of $y$ this gives
$$
\pr_{1!}\verk \ord_\infty((x,y)^*{_c\vartheta_\cE})
=\ell^rB_{2,c,r}^{\langle p(t)\rangle}(x).
$$
With diagram \eqref{ord-comm-diag} we finally get
$$
\ord_\infty\verk p_*({_c\vartheta_{\cE}})
=B_{2,c,r}^{\langle p(t)\rangle}\in \bbZ[Z_r\langle p(t)\rangle].
$$
\end{proof}

From Theorem \ref{ES-c-r-res} we will deduce  a formula
for the residue of the elliptic Soul\'e elements.

Recall the elliptic Soul\'e element
$$
_c e_{k}(t)=\mom^k_t(\cE\cS_c^{\langle t\rangle})\in H^1(Y(N),\TSym^k\sH(1))
$$
from Definition \ref{e-k-def} and consider 
$$
\frac{1}{N^k}{_ce}_k(t)=\wt\mom^k_t(\cE\cS_c^{\langle t\rangle})\in
H^1(Y(N),\Sym^k\sH_{\bbQ_\ell}(1)).
$$
In Definition \ref{residue-Qell-def} and \eqref{residue-Qell-def-variant}
we have defined the residue map
$$
\res_\infty:H^1(Y(N),\Sym^k\sH_{\bbQ_\ell}(1))\to 
H^0(\infty, \bbQ_\ell)\isom \bbQ_\ell.
$$
In the next theorem we identify $\cE[N]\isom (\bbZ/N\bbZ)^2$.
\begin{theorem}\label{res-of-ell-Soule}
Let $t=(a,b)\in \cE[N](Y(N))\setminus\{e\}$, then 
$$
\res_\infty(_c e_{k}(t))=
\frac{N^{k+1}}{k!(k+2)}(c^{2}B_{k+2}(\{\frac{a}{N}\})-
c^{-k}B_{k+2}(\{\frac{ca}{N}\}))
$$
In particular, if $c\equiv 1\mod{N}$ one gets
$$
\res_\infty(_c e_{k}(t))=
\frac{N^{k+1}}{k!(k+2)}\frac{c^{k+2}-1}{c^k}
B_{k+2}(\{\frac{a}{N}\}).
$$
\end{theorem}

\begin{proof}[Proof of Theorem \ref{res-of-ell-Soule}]
By  Proposition \ref{ell-soule-as-mom}, Lemma \ref{res-comparison} and
Theorem \ref{ES-c-r-res} one has
\begin{align*}
\res_\infty({_c e}_{k}(t))&=N^k\res_\infty(\wt\mom_t^k(\cE\cS^{\langle t\rangle}_c))\\
&=\frac{1}{k!}\mom_a^k(\res_\infty(\cE\cS^{\langle t\rangle}_c))\\
&=\frac{1}{k!}\mom_a^k(B_{2,c}^{\langle a\rangle})\\
&=\frac{N^{k+1}}{k!(k+2)}(c^{2}B_{k+2}(\{\frac{a}{N}\})-
c^{-k}B_{k+2}(\{\frac{ca}{N}\})),
\end{align*}
where the last equality is formula \eqref{mom-of-Bernoulli}.
\end{proof}
\begin{corollary}\label{eis-residue}
Let $S=Y(N)$ and consider the residue map
from 
$$
\res_\infty:H^1(Y(N),\Sym^k\sH_{\bbQ_\ell}(1))\to H^0(\infty, \bbQ_\ell)\isom \bbQ_\ell,
$$
then if $t=(a,b)\in (\bbZ/N\bbZ)^2\setminus\{(0,0)\}$ one has
$$
\res_{\infty}(\Eis^k_{\bbQ_\ell}(t))=\frac{-N^k}{k!(k+2)}B_{k+2}(\{\frac{a}{N}\}).
$$ 
\end{corollary}
\begin{proof}
This is Theorem \ref{pol-eis} together with Corollary \ref{res-of-ell-Soule}
in the case $c\equiv 1\mod{N}$. 
\end{proof}
\begin{remark}
The formula differs by a minus sign from the one in \cite{Huber-Kings99} as we have
a different uniformization of the elliptic curve. 
\end{remark}

%
%
\section{The evaluation of the cup-product construction for elliptic 
Soul\'e elements}
%
%

%
\subsection{A different description of the cup-product construction}
%

Consider over $\wh{Y}(N)_\infty$ the sheaf $\Sym^k\sH_{\bbQ_\ell}(1)$ and
the diagram
$$
\wh{Y}(N)_\infty\xrightarrow{j}\wh{X}(N)_\infty\xleftarrow{\infty}\infty.
$$
Recall from Corollary \ref{monodromy-sym} the isomorphisms
\begin{align*}
\bbQ_\ell(k+1)\isom \infty^*j_*\Sym^k\sH_{\bbQ_\ell}(1)&&\mbox{and}&&
\infty^*R^1j_*\Sym^k\sH_{\bbQ_\ell}(1)\isom \bbQ_\ell.
\end{align*}
The Leray spectral sequence for $Rj_*$ induces an exact sequence
\begin{multline}
0\to H^1(\wh{X}(N)_\infty,j_*\Sym^k\sH_{\bbQ_\ell}(1))\to
H^1(\wh{Y}(N)_\infty,\Sym^k\sH_{\bbQ_\ell}(1))\\
\xrightarrow{\res_\infty}
H^0(\infty,\bbQ_\ell)\to 0
\end{multline}
and we consider the Eisenstein class 
$$
\Eis_{\bbQ_\ell}^k(\psi)\in H^1(\wh{Y}(N)_\infty,\Sym^k\sH_{\bbQ_\ell}(1)).
$$
The next result gives a different description of the cup-product construction.
\begin{theorem}[\cite{Huber-Kings99} Theorem 2.4.1, \cite{Huber-pune}
Theorem 4.2.1]\label{Dir-as-evaluation}
Assume that $\res_\infty(\Eis_{\bbQ_\ell}^k(\psi))=0$ so that one can
consider
$$
\Eis_{\bbQ_\ell}^k(\psi)\in H^1(\wh{X}(N)_\infty,j_*\Sym^k\sH_{\bbQ_\ell}(1)).
$$
Then 
$$
\Dir_\ell(\psi)=\infty^*\Eis_{\bbQ_\ell}^k(\psi)
$$
in $H^1(\infty,\infty^*j_*\Sym^k\sH_{\bbQ_\ell}(1))\isom H^1(\infty,\bbQ_\ell(k+1))$.
\end{theorem}
Recall that 
$$
\Eis_{\bbQ_\ell}^k(\psi)=\sum_{t\in \cE[N]\setminus \{e\}}\psi(t)
\Eis_{\bbQ_\ell}^k(t)
$$
and note that it is not possible to 
evaluate the individual classes $\Eis_{\bbQ_\ell}^k(t)$ at $\infty$ as 
$\res_\infty(\Eis_{\bbQ_\ell}^k(t))\neq 0$.

The idea for the evaluation is as follows: Using 
Theorem \ref{pol-eis} we have 
$$
{_c{e}_k(t)}=
\wt\mom^k_t(\cE\cS_c^{\langle t\rangle})=
-N
(c^2\Eis_{\bbQ_\ell}^k(t)-c^{-k}\Eis_{\bbQ_\ell}^k({[c]t})).
$$
Although $\cE\cS_c^{\langle t\rangle}$ still can not be evaluated at
$\infty$, we will define an auxiliary class $\cB\cS_c^{\langle t\rangle}$
in Definition \ref{BS-def} which has the same residue as 
$\cE\cS_c^{\langle t\rangle}$. The difference 
$$
\cM\cE\cS_c^{\langle t\rangle}:=\cE\cS_c^{\langle t\rangle}-
\cB\cS_c^{\langle t\rangle}
$$
has then residue zero and can be evaluated at $\infty$. For
this we use
the description of $\cM\cE\cS_c^{\langle t\rangle}$ 
by an explicit function via the Kummer map. The evaluation 
at $\infty$ is then just the evaluation of the function at $q=0$,
where $q$ is the local parameter at $\infty$. We conclude 
by comparing the resulting function with the one defining the 
Soul\'e-Deligne classes.

%
\subsection{The auxiliary class $\cB\cS_c^{\langle t\rangle}$}
%

Recall from \eqref{p-r-def} the finite morphism
$$
p_r:\cE[\ell^r]\langle t\rangle\to Z_r\langle p(t)\rangle
$$
and recall that $\wh{Y}(N)_\infty=\Spec \bbQ(\zeta_N)((q^{1/N}))$.

On $\cE[\ell^r]\langle p(t)\rangle$ consider the function
$$
Bp_{2,c,r}^{\langle t\rangle}(x):=
\frac{N}{2}(c^2B_2(\{\frac{p_r(x)}{\ell^rN}\})-B_2(\{\frac{cp_r(x)}{\ell^rN}\}))
=\frac{1}{\ell^r}B_{2,c,r}^{\langle p(t)\rangle}(p_r(x)),
$$
which defines an element in $\Lambda_r(\sH_r\langle t\rangle)$, hence a
global section
$$
Bp_{2,c,r}^{\langle t\rangle}\in 
H^0(\wh{Y}(N)_\infty,\Lambda_r(\sH_r\langle t\rangle)).
$$ 
\begin{lemma}
The elements 
\begin{equation*}
Bp_{2,c,r}^{\langle t\rangle}\in
H^0(\wh{Y}(N)_\infty,\Lambda_r(\sH_r\langle t\rangle))
\end{equation*}
are norm-compatible, i.e., one can define
$$
Bp_{2,c}^{\langle t\rangle}:=
\prolim_r
Bp_{2,c,r}^{\langle t\rangle}\in
H^0(\wh{Y}(N)_\infty,\Lambda(\sH\langle t\rangle)).
$$
\end{lemma}
\begin{proof}
Consider the push-out by $p_1$
$$
\begin{CD}
0@>>> \cE[\ell]@>>>\cE[\ell^{r+1}N]@>>>\cE[\ell^rN]@>>> 0\\
@.@Vp_1VV@V\varrho VV@V=VV\\
0@>>> \bbZ/\ell\bbZ@>>>\wt\cE_{r+1}@>\sigma >>\cE[\ell^rN]@>>>0\\
@.@V=VV@VVV@VVV\\
0@>>> \bbZ/\ell\bbZ@>>>\bbZ/\ell^{r+1}N\bbZ@>>> \bbZ/\ell^rN\bbZ@>>>0.
\end{CD}
$$
This induces on the fibres over $t\in\cE[N]$
$$
\Lambda_{r+1}(\sH_{r+1}\langle t\rangle)\xrightarrow{\varrho_!}
\Lambda_{r+1}(\wt\sH_{r+1}\langle t\rangle)\xrightarrow{\sigma_!}
\Lambda_{r+1}(\sH_{r+1}\langle t\rangle)
$$
where $\wt\sH_{r+1}\langle t\rangle$ is the sheaf associated
to the fibre of $\wt\cE_{r+1}$ over $t$. As $Bp_{2,c,r+1}^{\langle t\rangle}$ is
a pull-back from $\bbZ/\ell^{r+1}N\bbZ$ the map $\varrho_!$ multiplies
the element with the cardinality of the fibres of $p_1$, which is $\ell$.
Application of $\sigma_!$ to $\ell Bp_{2,c,r+1}^{\langle t\rangle}$ gives
by the norm-compatibility of $B_{2,c,r+1}^{\langle p(t)\rangle}$ exactly
$Bp_{2,c,r}^{\langle t\rangle}$, which is the
desired result.
\end{proof}
\begin{lemma}\label{BS-def}
Let  $\eta_r$ be the invertible function on $\cE[\ell^r]\langle t\rangle$
$$
\eta_r(x):=(q^{1/N})^{Bp_{2,c,r}^{\langle t\rangle}(x)}=
q^{\frac{1}{\ell^rN}B_{2,c,r}^{\langle p(t)\rangle}(p_r(x))}.
$$
Then the class 
$$
\cB\cS_{c,r}^{\langle t\rangle}:=\partial_r(\eta_r)\in
H^1(\wh{Y}(N)_\infty,\Lambda_r(\sH_r\langle t\rangle)(1))
$$
is the image of 
$$
\partial_r(q^{1/N})\otimes Bp_{2,c,r}^{\langle t\rangle}\in
H^1(\wh{Y}(N)_\infty,\Lambda_r(1))\otimes 
H^0(\wh{Y}(N)_\infty,\Lambda_r(\sH_r\langle t\rangle))
$$
under the cup-product. In particular, one can define 
\begin{equation}
\cB\cS_c^{\langle t\rangle}:=\prolim_r\cB\cS_{c,r}^{\langle t\rangle}
\in H^1(\wh{Y}(N)_\infty, \Lambda(\sH\langle t\rangle)(1)).
\end{equation}
\end{lemma}

\begin{lemma}\label{BS-res}
One has 
$$
\res_\infty(\cB\cS_c^{\langle t\rangle})= B_{2,c}^{\langle p(t)\rangle}.
$$
\end{lemma}
\begin{proof}
The residue $\res_\infty$ factors through 
$p_{r!}:\Lambda_r(\sH_r\langle t\rangle)\to \Lambda_r(Z_r\langle p(t)\rangle)$
and by definition 
$$
p_{r!}(Bp_{2,c,r}^{\langle t\rangle})=\ell^r Bp_{2,c,r}^{\langle t\rangle}=
B_{2,c,r}^{\langle p(t)\rangle}.
$$
Using the cup-product representation of $\cB\cS_{c,r}^{\langle t\rangle}$
gives the desired result. 
\end{proof}
Consider the moment maps
$$
\mom_t^k:H^1(\wh{Y}(N)_\infty,\Lambda(\sH\langle t\rangle)(1))\to
H^1(\wh{Y}(N)_\infty,\TSym^k\sH(1)).
$$
\begin{definition}
We define 
\begin{equation*}
_c b_k(t):=\mom_t^k(\cB\cS_c^{\langle t\rangle})\in
H^1(\wh{Y}(N)_\infty,\TSym^k\sH(1)).
\end{equation*}
For a function $\psi:\cE[N]\setminus \{e\}\to \bbQ$ we let
\begin{equation*}
_c b_k(\psi):=\sum_{t\in \cE[N]\setminus \{e\}}\psi(t){_c} b_k(t).
\end{equation*}
\end{definition}

\begin{proposition}\label{BS-evaluation}
Let $\psi$ be a function such that
$\res_\infty(_c b_k(\psi))=0$, then
$$
_c b_k(\psi)=0
$$
in $H^1(\wh{Y}(N)_\infty,\TSym^k\sH(1))$.
\end{proposition}
\begin{proof}
From Lemma \ref{BS-def} we see that ${_c} b_k(t)$ is a cup-product
$$
{_c} b_k(t)=(\prolim_r\partial_r(q^{1/N}))\cup 
\mom_t^k(Bp_{2,c}^{\langle t\rangle}),
$$
where $\mom_t^k(Bp_{2,c}^{\langle t\rangle})\in 
H^0(\wh{Y}(N)_\infty,\TSym^k\sH)$. The map 
$p:\sH\to\bbZ_\ell$ induces an isomorphism
$$
H^0(\wh{Y}(N)_\infty,\TSym^k\sH)\isom 
H^0(\wh{Y}(N)_\infty,\bbZ_\ell)\isom \bbZ_\ell
$$
because of weight reasons. The image of $\mom_t^k(\sum_{t\in \cE[N]\setminus \{e\}}\psi(t)Bp_{2,c}^{\langle t\rangle})$
under this isomorphism is just $\res_\infty(_c b_k(\psi))$, which
is zero by assumption.
\end{proof}

%
\subsection{Evaluation at $\infty$ of the modified elliptic Soul\'e element}
%

We modify the elliptic Soul\'e element ${_c e_k(t)}$ by
subtracting the element ${_c} b_k(t)$. The resulting element has
no residue at $\infty$ and hence can be evaluated.

\begin{definition}\label{me-def}
We let 
$$
\cM\cE\cS_{c,r}^{\langle t\rangle}:=\cE\cS_{c,r}^{\langle t\rangle}-
\cB\cS_{c,r}^{\langle t\rangle}\in 
H^1(\wh{Y}(N)_\infty,\Lambda_r(\sH_r\langle t\rangle)(1))
$$
and $\cM\cE\cS_{c}^{\langle t\rangle}:=\prolim_r \cM\cE\cS_{c,r}^{\langle t\rangle}$. Define 
$\varepsilon_r:={_c\vartheta_{\cE}}\eta_r^{-1}\in \Gm(\cE[\ell^r]\langle t\rangle)$ so
that 
$$
\partial_r(\varepsilon_r)=\cM\cE\cS_{c,r}^{\langle t\rangle}.
$$
Let 
$$
_c\wt{me}_k(t):=
\wt\mom^k_t(\cM\cE\cS_{c}^{\langle t\rangle})
=\frac{1}{N^k}({_ce_k(t)}-{_cb_k(t)})
\in 
H^1(\wh{Y}(N)_\infty,\Sym^k\sH_{\bbQ_\ell}(1)).
$$
\end{definition}
By construction, the residue at $\infty$ of $_c\wt{me}_k(t)$ is zero:
\begin{lemma}\label{res-of-me}
One has $\res_\infty(\cM\cE\cS_{c}^{\langle t\rangle})=0$ hence
$$
\res_\infty(_c\wt{me}_k(t))=0,
$$
so that one can consider $_c\wt{me}_k(t)$ as a class in 
$$
H^1(\wh{X}(N)_\infty,j_*\Sym^k\sH_{\bbQ_\ell}(1)).
$$
\end{lemma}
\begin{proof}
This follows from Lemma \ref{res-comparison},
Theorem \ref{ES-c-r-res} and Lemma \ref{BS-res}.
\end{proof}
We now want to evaluate
$$
\infty^*(_c\wt{me}_k(t))\in H^1(\infty, \bbQ_\ell(k+1))
$$
in terms of Soul\'e-Deligne elements. Recall that over $\wh{Y}(N)_\infty$
one has an exact sequence
$$
0\to \mu_N\xrightarrow{\iota}\cE[N]\xrightarrow{p}\bbZ/N\bbZ\to 0.
$$
\begin{theorem}\label{me-evaluation} Let $t$ be a non-zero $N$-torsion section of $\cE$. 
Let $_c\wt{me}_k(t)$ be the element defined in \ref{me-def}. If $p(t)\neq 0$
one has $\infty^*(_c\wt{me}_k(t))=0$. If $p(t)=0$, $t$ is in the image
of $\iota$ and will be considered as an $N$-th root of
unity. Then the formula
\begin{multline*}
\infty^*(_c\wt{me}_k(t))=\\
\frac{1}{2k!N^k}\left(
c^2(\wt c_{k+1}(t)+(-1)^k\wt c_{k+1}(t^{-1}))+ c^{-k}(\wt c_{k+1}(ct)+(-1)^k\wt c_{k+1}(ct^{-1}))\right)
\end{multline*}
holds in $H^1(\infty,\bbQ_\ell(k+1))$.
\end{theorem}
The proof of this theorem is given in the next section. 

\begin{remark}\label{c-rem}
In fact one can show that in $H^1(\infty,\bbQ(k+1))$ the identity $\wt c_{k+1}(t^{-1})=(-1)^k\wt c_{k+1}(t)$ holds
(see for example \cite{Deligne-groupe-fondamental} 3.14.) but we do not 
need this fact.
\end{remark}
The consequences for the evaluation of the cup-product construction are as
follows. Identify $\cE[N]\isom (\bbZ/N\bbZ)^2$ and recall that 
$$
\Eis^k_{\bbQ_\ell}(\psi)=\sum_{(a,b)\in (\bbZ/N\bbZ)^2\setminus \{(0,0)\}}
\psi(a,b){\Eis^k_{\bbQ_\ell}(a,b)}.
$$
\begin{corollary}\label{Dir-formula}
With the above notations suppose that $\psi$ is a function with
$\res_\infty(\Eis^k_{\bbQ_\ell}(\psi))=0$.
Then
$$ 
\Dir_\ell(\psi)=\infty^*(\Eis^k_{\bbQ_\ell}(\psi))=\frac{-1}{k!N}\sum_{b\in \bbZ/N\bbZ\setminus \{0\}}
\psi(0,b)\wt c_{k+1}(\zeta_N^b)
$$
where $\zeta_N=e^{2\pi i/N}$.
\end{corollary}
\begin{proof}
By assumption we have $\res_\infty(\Eis^k_{\bbQ_\ell}(\psi))=0$ which
implies $\res_\infty({_cb_k(\psi)})=0$.
It follows from Proposition \ref{BS-evaluation} 
that 
$$
\frac{1}{N^k}\infty^*(_ce_k(\psi))=\sum_{(a,b)\in (\bbZ/N\bbZ)^2\setminus \{(0,0)\}}
\psi(a,b)
\infty^*(_c\wt{me}_k(a,b)).
$$
We now apply Theorem \ref{me-evaluation} and observe that 
$\infty^*(_c\wt{me}_k(a,b))=0$, if $a\neq 0$. 
By Lemma \ref{Eis-parity} we can also assume right away that 
$\psi(-t)=(-1)^k\psi(t)$. Then one has
$$
\sum_{b\in \bbZ/N\bbZ\setminus \{0\}}
\psi(0,b)\wt c_{k+1}(\zeta_N^b)=
\sum_{b\in \bbZ/N\bbZ\setminus \{0\}}
\psi(0,b)(-1)^k\wt c_{k+1}(\zeta_N^{-b})
$$
and
$$
\sum_{b\in \bbZ/N\bbZ\setminus \{0\}}
\psi(0,b)\wt c_{k+1}(\zeta_N^{cb})=
\sum_{b\in \bbZ/N\bbZ\setminus \{0\}}
\psi(0,b)(-1)^k\wt c_{k+1}(\zeta_N^{-cb})
$$
by substituting $b\mapsto -b$. If we use this in the formula of
Theorem \ref{me-evaluation} in the case of $c\equiv 1\mod{N}$ we get
$$
\frac{1}{N^k}\infty^*(_ce_k(\psi))=
\frac{c^2-c^{-k}}{k!N^k}\sum_{b\in \bbZ/N\bbZ\setminus \{0\}}
\psi(0,b)
\wt c_{k+1}(\zeta_N^b).
$$
On the other hand by Theorem \ref{pol-eis} for $c\equiv 1\mod{N}$
$$
\frac{1}{N^k}\infty^*(_ce_k(\psi))=
\frac{-(c^2-c^{-k})}{N^{k-1}}\Eis^k_{\bbQ_\ell}(\psi),
$$
which gives the desired result.
\end{proof}

%
\subsection{Proof of Theorem \ref{me-evaluation}}
%
We need to introduce some more notation.
Over $\wh{Y}(N)_\infty$ one has
$$
\begin{CD}
\mu_{\ell^rN}@>\iota_r>>\cE[\ell^rN]@>p_r>> \bbZ/\ell^rN\bbZ\\
@V[\ell^r]VV@VV[\ell^r]V@VV[\ell^r]V\\
\mu_N@>\iota>> \cE[N]@>p>>\bbZ/N\bbZ.
\end{CD}
$$
\begin{definition}
We denote by 
$\mu_{\ell^r}\langle t\rangle\xrightarrow{\iota_r}\cE[\ell^r]\langle t\rangle$
the fibre over the $N$-torsion section $t$. 
We let $\sT_r\langle t\rangle$ be the sheaf associated to 
$\mu_{\ell^r}\langle t\rangle$ and define
$$
\sT\langle t\rangle:=\prolim_r\sT_r\langle t\rangle.
$$
In the case $t=0$ we write $\sT:=\sT\langle 0\rangle$.
\end{definition}
Note that $\mu_{\ell^r}\langle t\rangle$ is empty if $p(t)\neq 0$.
The maps $\iota_r$
induce a map of sheaves
\begin{equation}
\iota_!:\Lambda(\sT\langle t\rangle)\to \Lambda(\sH\langle t\rangle).
\end{equation}
On the other hand, pull-back by $\iota_r$ gives a map of sheaves
\begin{equation}
\iota^*:\Lambda(\sH\langle t\rangle)\to \Lambda(\sT\langle t\rangle)
\end{equation}
which is a splitting of $\iota_!$. On the other hand the maps 
$p_{r,t}:\cE[\ell^r]\langle t\rangle\to Z_r\langle p_r(t)\rangle$
give rise to 
\begin{equation}
p_!:\Lambda(\sH\langle t\rangle)\to \Lambda(Z\langle p(t)\rangle).
\end{equation}
\begin{proposition}
The morphisms $\iota_!$ and $p_!$ induce isomorphisms
$$
\Lambda(\sT)\isom \infty^*j_*\Lambda(\sH)
$$
and
$\infty^*R^1j_*\Lambda(\sH)\isom\Lambda(Z)$.
\end{proposition}
\begin{proof}
Let $I(\sT)$, $I(\sH)$ and $I(Z)$ be the augmentation ideals of 
$\Lambda(\sT)$, $\Lambda(\sH)$ and $\Lambda(Z)$ respectively and 
$\Lambda(\sT)^{(k)}$ etc. the quotient by the $k+1$-power of the
augmentation ideal. Then by induction on $k$
and Corollary
\ref{monodromy-sym} one has a commutative diagram 
$$
\begin{CD}
0@>>> \infty^*j_*\Sym^k\sH@>>> \infty^*j_*\Lambda(\sH)^{(k)}@>>> \infty^*j_*\Lambda(\sH)^{(k-1)} \\
@.@AA\isom A@AA\iota_! A@AA\isom A\\
0@>>>\bbZ_\ell(k)@>>>\Lambda(\sT)^{(k)}@>>> \Lambda(\sT)^{(k-1)}@>>> 0
\end{CD}
$$
which implies that the injective morphism $\iota_!$ is also surjective.
In the same way one shows $p_!:\infty^*R^1j_*\Lambda(\sH)\isom\Lambda(Z)$.
\end{proof}
With this result the Leray spectral sequence for $Rj_*$ gives
$${\small
\xymatrix{
&&H^1(\wh{Y}(N)_\infty, \Lambda(\sH\langle t\rangle)(1))\ar[d]_{[N]_!}\ar[r]&
H^0(\infty, \Lambda(Z\langle t\rangle))\ar[d]_{[N]_!}\\
0\ar[r]&H^1(\wh{X}(N)_\infty,j_*\Lambda(\sH)(1))\ar[r]&
H^1(\wh{Y}(N)_\infty, \Lambda(\sH)(1))\ar[r]^/.7em/{\res_\infty}&
H^0(\infty, \Lambda(Z)).
}
}$$
Recall that we want to compute 
$$
\infty^*(_c\wt{me}_k(t))=\infty^*\wt\mom_t^k(\cM\cE\cS_c^{\langle t\rangle})=\infty^*
\wt\mom^k\verk[N]_!(\cM\cE\cS_c^{\langle t\rangle}).
$$
From Lemma \ref{res-of-me} we get that 
\begin{equation}
[N]_!(\cM\cE\cS_c^{\langle t\rangle})\in 
H^1(\wh{X}(N)_\infty,j_*\Lambda(\sH)(1)).
\end{equation}
Consider the commutative diagram
$$
\CD
H^1(\wh{X}(N)_\infty, \Lambda(\sT)(1))@>\iota_!>>
H^1(\wh{X}(N)_\infty, j_*\Lambda(\sH)(1))
\\
@V\wt\mom^k VV@VV\wt\mom^k V\\
H^1(\wh{X}(N)_\infty,\bbQ_\ell(k+1))@>>>
H^1(\wh{X}(N)_\infty, j_*\Sym^k\sH_{\bbQ_\ell}(1))\\
@VV{\infty^*}V@VV{\infty^*}V\\
H^1(\infty,\bbQ_\ell(k+1))@=
H^1(\infty,\bbQ_\ell(k+1)).
\endCD
$$
As $\iota_!$ has the splitting $\iota^*$ it follows that we have
\begin{equation}\label{wtme-formula}
\infty^*(_c\wt{me}_k(t))=\infty^*\wt\mom^k\verk\iota^*\verk
[N]_!(\cM\cE\cS_c^{\langle t\rangle})=\wt\mom^k_t\verk\infty^*\iota^*
(\cM\cE\cS_c^{\langle t\rangle}).
\end{equation}
At finite level we have
$$
\begin{CD}
\Gm(\cE[\ell^r]\langle t\rangle)@>\iota^*>>\Gm(\mu_{\ell^r}\langle t\rangle)\\
@V\partial_r VV@VV\partial_r V\\
H^1(\wh{Y}(N)_\infty, \Lambda_r(\sH_r\langle t\rangle)(1))@>\iota^*>>
H^1(\wh{Y}(N)_\infty, \Lambda_r(\mu_{\ell^r}\langle t\rangle)(1))
\end{CD}
$$
which implies that we have with the notation in Definition \ref{me-def}
$$
\iota^*
(\cM\cE\cS_{c,r}^{\langle t\rangle})=\partial_r(\iota^*(\varepsilon_r)).
$$
\begin{lemma}
The function $\iota^*(\varepsilon_r)\in \Gm(\mu_{\ell^r}\langle t\rangle)$
extends to a function on  $\mu_{\ell^r}\langle t\rangle$ over all of 
$\wh{X}(N)_\infty$ (also denoted by $\iota^*(\varepsilon_r)$).
The special fibre $\infty^*\iota^*(\varepsilon_r)$ is the function
\begin{align*}
\infty^*\iota^*(\varepsilon_r):\mu_{\ell^r}\langle t\rangle&\to\Gm\\
\beta&\mapsto (-\beta)^{\frac{c-c^2}{2}} \frac{(1-\beta)^{c^2}}{(1-\beta^c)}.
\end{align*}
\end{lemma}
\begin{proof}
Note first that there is nothing to show if $p(t)\neq 0$. 
So we assume that $p(t)=0$. That $\iota^*(\varepsilon_r)$ extends 
can be checked after a base extension which adjoins all 
$\ell^rN$-th roots of unity. Then by definition and Corollary
\ref{theta-evaluation}
$\iota^*(\varepsilon_r)$ has the form
$$
\iota^*(\varepsilon_r)(\beta)=
(-\beta)^{\frac{c-c^2}{2}} \frac{(1-\beta)^{c^2}}{(1-\beta^c)}
\frac{\wt\gamma_q(\beta)^{c^2}}{\wt\gamma_q(\beta^c)},
$$
where $\beta \in\mu_{\ell^r}\langle t\rangle$ and 
$$
\wt\gamma_q(\beta)=\prod_{n> 0}(1-q^n\beta)(1-q^n\beta^{-1}).
$$
From this formula it is clear that $\iota^*(\varepsilon_r)$ makes sense
for $q=0$, i.e., extends to $\mu_{\ell^r}\langle t\rangle$ over all of 
$\wh{X}(N)_\infty$. Putting $q=0$ gives the explicit form of
$\infty^*\iota^*(\varepsilon_r)$ as claimed.
\end{proof}
\begin{lemma} Assume $p(t)=0$ so that $t$ is in the image of $\iota$
and will be considered as an $N$-th root of unity.  
In $H^1(\infty, \Lambda_r(\sT_r\langle t\rangle)(1))$ one has
the identity
$$
2\infty^*\iota^*(\cM\cE\cS_{c,r}^{\langle t\rangle})=
\cC\cS_{c,r}^{\langle t\rangle}
+[-1]^*\cC\cS_{c,r}^{\langle t^{-1}\rangle},
$$
which gives in the limit 
$$
2\infty^*\iota^*(\cM\cE\cS_{c}^{\langle t\rangle})=
\cC\cS_{c}^{\langle t\rangle}
+[-1]^*\cC\cS_{c}^{\langle t^{-1}\rangle}
$$
in $H^1(\infty, \Lambda(\sT\langle t\rangle)(1))$.
\end{lemma}
\begin{proof}
A direct computation gives 
$$
\left((-\beta)^{\frac{c-c^2}{2}} \frac{(1-\beta)^{c^2}}{(1-\beta^c)}\right)^2=
\frac{(1-\beta)^{c^2}(1-\beta^{-1})^{c^2}}{(1-\beta^c)(1-\beta^{-c})}
$$
which implies
$$
\infty^*\iota^*(\varepsilon_r^2)={_c\Xi}\cdot(_c\Xi\verk[-1]),
$$
where $_c\Xi$ is the function defined in \eqref{Xi-def}. Thus, applying 
the Kummer map $\partial_r$ one has
\begin{align*}
2\infty^*\iota^*(\cM\cE\cS_{c,r}^{\langle t\rangle})&=
\cC\cS_{c,r}^{\langle t\rangle}+[-1]^*\cC\cS_{c,r}^{\langle t^{-1}\rangle}
\end{align*}
because 
$[-1]:\mu_{\ell^r}\langle t\rangle\isom \mu_{\ell^r}\langle t^{-1}\rangle$.
\end{proof}
To conclude the proof of Theorem \ref{me-evaluation}, we have to
compute 
$$
\wt\mom^k_t(\cC\cS_{c}^{\langle t\rangle}+[-1]^*\cC\cS_{c}^{\langle t^{-1}\rangle}).
$$
In the definition of $\wt\mom_t^k$ we use the inverse of the identification
$\bbQ_\ell(k)\isom\Sym^k\bbQ_\ell(1)\isom \TSym^k\bbQ_\ell(1)\isom \bbQ(k)$
which maps $1\mapsto k!$. Note also that, as in the proof of Proposition
\ref{mom-of-S}, one has
$$
\mom_t^k([-1]^*\cC\cS_{c}^{\langle t^{-1}\rangle})=(-1)^k\wt c_{k+1}(t^{-1}).
$$
With formula \eqref{wtme-formula} and Proposition \ref{mom-of-S} 
one now gets:
\begin{corollary}
One has in $H^1(\infty,\bbQ_\ell(k+1))$ the identity
\begin{multline*}
\infty^*(_c\wt{me}_k(t))=\\
\frac{1}{2k!N^k}\left(
c^2(\wt c_{k+1}(t)+(-1)^k\wt c_{k+1}(t^{-1}))+ c^{-k}(\wt c_{k+1}(ct)+(-1)^k\wt c_{k+1}(ct^{-1}))\right).
\end{multline*}
\end{corollary}
This proves Theorem \ref{me-evaluation}.
\bibliography{kings}
\bibliographystyle{alpha}
\end{document}